\documentclass[11pt]{article}
\usepackage{amsfonts}
\usepackage{amssymb}
\usepackage{amsthm}
\usepackage{amsmath}
\usepackage{graphicx}
\usepackage{empheq}
\usepackage{indentfirst}
\usepackage{cite}
\usepackage{mathrsfs}
\usepackage{graphics}
\usepackage{enumerate}
\usepackage{color}

 \newtheorem{theorem}{Theorem}[section]
 
 \newtheorem{lemma}{Lemma}[section]
 \newtheorem{proposition}{Proposition}[section]

 \newtheorem{remark}{Remark}[section]
 \numberwithin{equation}{section}

\def\u{\tilde{u}}
\def\v{\tilde{v}}
\def\M{\mathcal{M}}

\allowdisplaybreaks[4]

\def\u{\tilde{u}}

\newcommand{\beq}{\begin{equation}}
\newcommand{\eeq}{\end{equation}}
 \def\non{\nonumber }
\def\bea{\begin{eqnarray}}
\def\eea{\end{eqnarray}}

\def\L{\mathcal{L}}
\begin{document}
\title{Global Stability of Keller--Segel Systems in Critical Lebesgue Spaces}
\author{Jie Jiang\thanks{Wuhan Institute of Physics and Mathematics, Chinese Academy of Sciences,
Wuhan 430071, HuBei Province, P.R. China,
\textsl{jiang@wipm.ac.cn}.}}

\date{\today}

\maketitle

\begin{abstract}This paper is concerned with the initial-boundary value problem for the  classical Keller--Segel system 
\begin{equation}
	\begin{cases}\label{chemo0}
		\rho_t-\Delta \rho=-\nabla\cdot(\rho\nabla c),\qquad &x\in\Omega,\;t>0\\
		\gamma c_t-\Delta c+c= \rho,\qquad &x\in\Omega,\;t>0\\
		\frac{\partial \rho}{\partial\nu}=\frac{\partial c}{\partial\nu}=0,\qquad &x\in\partial\Omega,\;t>0\\
		\rho(x,0)=\rho_0(x),\;\;\gamma c(x,0)=\gamma c_0(x),\qquad & x\in\Omega
	\end{cases}
	\end{equation}
 in  a bounded domain $\Omega\subset\mathbb{R}^d$ with $d\geq2$, where $\gamma=0$ or $1$. We study the existence of non-trivial global classical solutions near the spatially homogeneous equilibria $\rho= c\equiv\M>0$ with $\M$ being any given large constant which is an open problem proposed in \cite[p. 1687]{BBTW15}. More precisely, we prove that if $0<\M<1+\lambda_1$ with $\lambda_1$ being the first positive eigenvalue of the Neumann Laplacian operator, one can find $\varepsilon_0>0$ such that for all suitable regular initial data $(\rho_0,\gamma c_0)$ satisfying \begin{equation}\label{assumpt1}
	\frac{1}{|\Omega|}\int_\Omega \rho_0dx-\M=\gamma\left(\frac{1}{|\Omega|}\int_\Omega c_0dx-\M\right)=0
\end{equation}and \begin{equation}\label{assumpt2}\|\rho_0-\M\|_{L^{d/2}(\Omega)}+\gamma\|\nabla c_0\|_{L^d(\Omega)}<\varepsilon_0,\end{equation} problem \eqref{chemo0} possesses a unique global classical solution which is bounded and converges to the trivial state $(\M,\M)$ exponentially as time goes to infinity. The key step of our proof lies in deriving certain delicate $L^p-L^q$ decay  estimates for the semigroup associated with the corresponding linearized system of \eqref{chemo0} around the constant steady states. It is well-known that  classical solution to system \eqref{chemo0} may blow up in finite or infinite time when the conserved total mass $m\triangleq\int_\Omega \rho_0 dx$ exceeds some threshold number if $d=2$ or for arbitrarily small mass if $d\geq3$, while our result links the dynamics of solutions explicitly with a new geometric quantity $|\Omega|$, i.e.,  non-trivial classical solutions starting from initial data  satisfying \eqref{assumpt1}-\eqref{assumpt2} with arbitrarily large total mass $m$ exists globally provided that $|\Omega|$ is large enough such that $m<(1+\lambda_1)|\Omega|$.

{\bf Keywords}: Chemotaxis, Keller--Segel model, global solutions, global stability.\\
\end{abstract}
\section{Introduction}
In this paper, we study the initial-boundary value problem for the following classical  Keller--Segel system of chemotaxis:
\begin{equation}
	\begin{cases}\label{chemo1}
		\rho_t-\Delta \rho=-\nabla\cdot(\rho\nabla c),\qquad &x\in\Omega,\;t>0\\
		\gamma c_t-\Delta c+c= \rho,\qquad &x\in\Omega,\;t>0\\
		\frac{\partial \rho}{\partial\nu}=\frac{\partial c}{\partial\nu}=0,\qquad &x\in\partial\Omega,\;t>0\\
		\rho(x,0)=\rho_0(x),\;\;\gamma c(x,0)=\gamma c_0(x),\qquad & x\in\Omega
	\end{cases}
\end{equation}where $\Omega\subset\mathbb{R}^d$ with $d\geq2$ is a bounded domain with smooth boundary. Here, $\rho$ and $c$ denote the density of  cells and the concentration of chemical signal, respectively. $\gamma\geq0$ is a given constant; when $\gamma=0$, \eqref{chemo1} is reduced to an elliptic--parabolic system which is usually called a simplified Keller--Segel model in the existing literature.

A well-known fact of the Keller--Segel model \eqref{chemo1} is that classical solutions with large initial data may blow up when dimension $d\geq2$ (see \cite{Win10,BBTW15,Nagai98, Nagai00} and references cited therein). In particular, a critical-mass phenomenon exists in the two-dimensional case. More precisely, if the conserved total mass of cells $m\triangleq\int_\Omega \rho_0 dx$ is lower than certain number, then global classical solution exists and remains bounded for all times; otherwise, it may blow up in finite or infinite time. It was observed that if $\gamma=0$, the threshold number is $4\pi$ for any bounded domain and is $8\pi$ for a disk of any radius.

On the other hand, since $\rho=c\equiv \M$ with any positive number $\M$ is a spatially homogeneous steady solution, an open problem proposed in a recent survey \cite[p. 1687]{BBTW15} is that for any given initial data $(\rho_0,c_0)$ sufficiently close to $\M$, whether we can get a non-trivial global classical  solution which is bounded for all times. The present contribution is devoted to this problem and gives a partially affirmative answer. More precisely, we prove that if $0<\M<1+\lambda_1$ with $\lambda_1$ being the first positive eigenvalue of the Neumann Laplacian operator, one can find $\varepsilon_0>0$ such that for all suitable regular initial data $(\rho_0,\gamma c_0)$ satisfying \begin{equation}\label{restrict}
	\frac{1}{|\Omega|}\int_\Omega \rho_0dx-\M=\gamma\left(\frac{1}{|\Omega|}\int_\Omega c_0dx-\M\right)=0
\end{equation}and $\|\rho_0-\M\|_{L^{d/2}(\Omega)}+\gamma\|\nabla c_0\|_{L^d(\Omega)}<\varepsilon_0$, problem \eqref{chemo1} possesses a unique global classical solution which is bounded and converges to $(\M,\M)$ exponentially as time goes to infinity. 

Observing that the conserved total mass $m=\int_\Omega\rho_0dx=\M|\Omega|<(1+\lambda_1)|\Omega|$, our result  indicates a {\it new} observation that  classical solution can be obtained globally starting from suitable initial data of {\it arbitrarily large} total mass $m$ provided that the area $|\Omega|$ is large, correspondingly. Note that due to the existing results, the threshold number is $8\pi$ when $\Omega=\mathcal{B}$ being a disk in $\mathbb{R}^2$ for the case $\gamma=0$, no matter how large the radius is. In this respect, we rigorously prove that globally bounded nontrivial classical solution exists with any over-$8\pi$ total mass for $\Omega=\mathcal{B}$ if the radius is large enough. We also note that in \cite{Lorz}, numerical evidence of existence of global classical solution with total mass above the critical mass $8\pi$ was showed for a simplified Keller--Segel-Stokes system with zero Dirichlet boundary conditions for the chemical concentration $c$, fluid velocity $u$ and Neumann boundary condition for the cell density $\rho$, respectively. However, there was no analytical proof available  and hence it is unknown whether the presence of Stokes fluid plays an essential role in their example since fluid advection had already been conjectured to regularize singular nonlinear dynamics \cite{KX16}. In this regard and to the best of our knowledge, the present contribution provides the first example with proof for the existence of non-trivial global classical solution with large total mass for the classical Keller--Segel system.

 It is worth mentioning that in \cite{Win13} (aslo \cite{Win10}), when $d\geq3$ and in the radial setting, Winkler proved that arbitrary small perturbations of any initial data may immediately produce blow-up when the considered norm is chosen in $L^p\times W^{1,2}$ with $p\in(1,\frac{2d}{d+2})$. Due to their result, for any $\varepsilon$, one can always find $(\rho_{0\varepsilon},c_{0\varepsilon})$ satisfying $\|\rho_{0\varepsilon}-\M\|_{L^p}+\|c_{0\varepsilon}-\M\|_{W^{1,2}}\leq \varepsilon$ such that the solution starting from $(\rho_{0\varepsilon},c_{0\varepsilon})$  blows up in finite time. We remark that there is no contradiction with our results since on the one hand we have the restriction \eqref{restrict} on the initial data and on the other hand, the metric space $L^{d/2}\times W^{1,d}$ under consideration in our result is more regular and hence smaller than $L^p\times W^{1,2}$ when $d\geq3.$
 
 In addition, we would like to point out that the metric space $ L^{d/2}\times \dot W^{1,d}$ is a scaling-invariant space for the Keller--Segel system. To see this point, we observe that system \eqref{chemo1} with the second equation replaced by $\gamma c_t-\Delta c=\rho$ has the following property (taking $\Omega=\mathbb{R}^d$): if $(\rho,c)$ is a solution to \eqref{chemo1}, then the pair $(\rho_{\lambda}, c_{\lambda})$ given by
		\begin{equation}
			\rho_{\lambda}(x,t)=\lambda^2 \rho(\lambda x,\lambda^2 t),\quad c_{\lambda}(x,t)= c(\lambda x,\lambda^2 t),\quad \forall\,\lambda>0,
		\end{equation}	
is also a solution. Then we easily verify that the norm of $(\rho_\lambda, c_\lambda)$ in $L^\infty(0,T; L^{d/2}(\Omega))\times L^\infty(0,T; \dot W^{1,d}(\Omega))$ is scaling-invariant and thus we call  $ L^{d/2}(\Omega)\times \dot W^{1,d}(\Omega)$ a scaling-invariant space (or a critical space). Some discussions on blow-up criteria in sub-critical spaces for Keller--Segel models can be found in \cite{BBTW15, JWZ18, Xiang15}.

To formulate our results, we need to introduce some notion and notations. For any $\M\geq0,$ denote $L^p_\M(\Omega)$ ($1\leq p<\infty$) the closed convex subset of $L^p(\Omega)$ satisfying $\frac{1}{|\Omega|}\int_\Omega w dx=\M$ with $w\in L^p(\Omega).$ Note that if $\M=0$, $L^p_0(\Omega)$ is a Banach spaces and the following Poincar\'e's inequality holds:
\begin{equation}
	\|w\|_{L^p(\Omega)}\leq c\|\nabla w\|_{L^p(\Omega)},\qquad\text{for all}\;\;w\in L^p_0(\Omega)
\end{equation}and we denote $\lambda_1$  the first positive eigenvalue of the Neumann Lapacian operator such that
\begin{equation}\label{poin}
	\lambda_1\|w\|_{L^2(\Omega)}^2\leq \|\nabla w\|^2_{L^2(\Omega)} ,\qquad\text{for all}\;\;w\in L^2_0(\Omega).
\end{equation}

Now we are in a position to state our main results. The first one is concerned with global stability for the elliptic--parabolic Keller--Segel system with $\gamma=0.$
\begin{theorem}\label{TH0}
	Let $d\geq2.$ For any given constants $0<\M<1+\lambda_1$ and $\lambda'<\mu_0\triangleq\lambda_1(1-\frac{\M}{1+\lambda_1})$, there exists $\varepsilon_0>0$ depending on  $\lambda'$ and $\Omega$ such that for any non-negative initial datum $\rho_0\in C(\overline{\Omega})\cap L^1_\M(\Omega)$ satisfying  $\|\rho_0-\M\|_{L^{d/2}(\Omega)}\leq\varepsilon_0$, system \eqref{chemo1} has a unique global classical solution such that
	\begin{equation}
		\|\rho(\cdot, t)-\M\|_{L^\infty(\Omega)}+\|\nabla c(\cdot, t)\|_{L^\infty(\Omega)}\leq Ce^{-\lambda' t}\qquad\text{for all}\;\;t\geq1
	\end{equation}with some $C>0$.
\end{theorem}

The other result is about the doubly parabolic case with $\gamma=1$ which is given below.
\begin{theorem}\label{TH1}
		Let $d\geq2.$ For any given constants $0<\M<1+\lambda_1$ and $\mu'<\mu_1\triangleq\lambda_1-\frac12\bigg(\sqrt{4\lambda_1\M+1}-1\bigg)>0$, there exists $\varepsilon_0>0$ depending on $\M$, $\mu'$ and $\Omega$ such that for any non-negative initial data $(\rho_0,c_0)\in \left(C(\overline{\Omega})\cap L^1_\M(\Omega)\right)\times \left(C^1(\overline{\Omega})\cap L^1_\M(\Omega)\right)$ satisfying $\partial_\nu c_0=0$ on $\partial\Omega$ and  $\|\rho_0-\M\|_{L^{d/2}(\Omega)}+\|\nabla c_0\|_{L^d(\Omega)}\leq\varepsilon_0$, system \eqref{chemo1} has a unique global classical solution such that
	\begin{equation}
		\|\rho(\cdot,t)-\M\|_{L^\infty(\Omega)}+\|\nabla c(\cdot,t)\|_{L^\infty(\Omega)}\leq Ce^{-\mu' t}\qquad\text{for all}\;\;t\geq1
	\end{equation} with some $C>0.$
\end{theorem}

The case $\M=0$  corresponds to the small-initial data results established for the Keller--Segel system in bounded domains in \cite{Win10,Cao} and for the Keller--Segel--Navier--Stokes system in $\mathbb{R}^d$ in \cite{KMS16}. The former two papers realized the proof by a one-step contradiction argument while in the other one it was proved based on the implicit function theorem.  The classical $L^p-L^q$ decay estimates for the heat semigroup $e^{t\Delta}$  plays an essential role in both cases since the first equation of \eqref{chemo1} can be regarded as a heat equation for $\rho$ with a quadratic perturbation $-\nabla\cdot(\rho\nabla c)$. With small initial data, the solution to the nonlinear equation should behave like a solution to the heat equation with higher order perturbations.

In the case $\M>0$, the key step consists in the analysis of the corresponding linearized system around $(\M,\M)$ and there comes certain new  difficulties. To see this point, we introduce the system for the reduced quantities, i.e., for any given $\M>0,$ let $u=\rho-\M$ and $v=c-\M$ be the reduced cell density  and chemical concentration, respectively. It follows that
\begin{equation}
	\begin{cases}\label{chemo2}
		u_t-\Delta u+\M\Delta v=-\nabla\cdot(u\nabla v),\qquad &x\in\Omega,\;t>0\\
		\gamma v_t-\Delta v+v= u,\qquad &x\in\Omega,\;t>0\\
		\frac{\partial u}{\partial\nu}=\frac{\partial v}{\partial\nu}=0,\qquad &x\in\partial\Omega,\;t>0\\
		u(x,0)=u_0(x)=\rho_0-\M,\;\;\gamma v(x,0)=\gamma v_0(x)=\gamma(c_0-\M),\qquad & x\in\Omega.
	\end{cases}
\end{equation}
Now our problem transforms to the existence of global solutions with small initial data for system \eqref{chemo2} which can be regarded as a quadratic perturbation $-\nabla \cdot(u\nabla v)$ of its corresponding linearized system which reads:
\begin{equation}
\begin{cases}\label{chemo2l}
\u_t-\Delta \u+\M\Delta \v=0,\qquad &x\in\Omega,\;t>0\\
\gamma \v_t-\Delta \v+\v= \u,\qquad &x\in\Omega,\;t>0\\
\frac{\partial \u}{\partial\nu}=\frac{\partial \v}{\partial\nu}=0,\qquad &x\in\partial\Omega,\;t>0.
\end{cases}
\end{equation}
The analysis of the decay behavior of solutions to the above linearized system becomes quite important and constitutes the major part of the present paper.

Note that {\it cross diffusion} appears when $\M>0$ even in the linearized system \eqref{chemo2l}, while when $\M=0$, the corresponding linearized system is much simpler since it contains a decoupled heat equation for $\u$ (or $\rho$) that is why the proof in \cite{Win10,Cao,KMS16} strongly relies on the decay estimates for $e^{t\Delta}.$ However, this is insufficient in our case due to the presence of cross diffusion coupling which brings some severe difficulties in the analysis especially for the doubly parabolic case $\gamma=1$. In order to study the decay behavior of solutions to \eqref{chemo2l}, we first use perturbation theory for semigroups and some delicate energy estimates in Hilbert spaces to derive certain exponentially decay estimates for $(\u,\nabla \v)$ in $L^2_0(\Omega)\times L^2_0(\Omega)$ with explicit decay rates (see Lemma \ref{exdlem0} and Lemma \ref{lmexdecaypp}).

 Since we aim to establish global stability in the scaling-invariant Lebesgue spaces which as we mentioned above is $L^{d/2}(\Omega)\times L^d(\Omega)$ for $(\rho,\nabla c)$, what we need next is deriving  similar $L^p-L^q$ decay estimates  for the semigroup associated with the linearized system \eqref{chemo2l}, denoted in the sequel by $e^{t\mathcal{L}}$ for $\gamma=0$ and $e^{t\mathcal{A}}$ for $\gamma=1$, respectively. More precisely, we need  decay estimates of $e^{t\mathcal{A}}(u_0,v_0)$ in $L^p$ in terms of $\|u_0\|_{L^{d/2}}+\|\nabla v_0\|_{L^d}$. The proof is nontrivial due to  the cross diffusion coupling as well as the difference in the critical Lebesgue exponents between $u_0$ and $\nabla v_0$. Exploiting the Gronwall inequalities, the well-known $L^p-L^q$ estimates for the Neumann heat semigroups, the obtained exponentially decay properties and more importantly, by a delicate successive iteration argument carried out simultaneously with respect to time and Lebesgue exponents, we successfully establish the desired decay estimates (see Lemma \ref{keylem1} and Lemma \ref{Cor}). With this at hand, we can finish our proof  either by adapting the one-step contradiction argument in \cite{Win10,Cao} or by implicit function theorem as done in \cite{KMS16}. Since the local well-posedeness and blow-up criteria in sub-critical spaces for classical solutions is well-known (Lemma \ref{loex1}), we find it more convenient to discuss in the former way.

We point out that  assumption $0<\M<1+\lambda_1$ or equivalently, $0<m=\M|\Omega|<(1+\lambda_1)|\Omega|$ is necessary for the exponentially decay property of solutions for \eqref{chemo2l} in $L^2_0(\Omega)\times H^1(\Omega)\cap L^2_0(\Omega)$ which could be easily verified by similar linear stability analysis as done in  \cite{KS70}. For the case $\Omega=\mathcal{B}$ being a disk of any radius in $\mathbb{R}^2$, it can be calculated that $\lambda_1|\mathcal{B}|\approx 1.84118^2\pi\approx3.39\pi<8\pi$ (cf. \cite{Ash}). As is well known, for $\gamma=0$ and a disk with any radius, any initial datum with total mass less than $8\pi$ should develop a globally bounded classical solution. However due to our result, we need the radius to be large when the total mass is greater than $\lambda_1|\mathcal{B}|\approx3.39\pi$. The reason is that our argument strongly relies on the exponentially decay property of the linearized system (also the solution obtained for the original system has an exponentially decay property) and as we mentioned above, a strong condition $0<\M<1+\lambda_1$ is then needed. Nevertheless, this smallness condition links the dynamics of solutions explicitly with the geometric quantity $|\Omega|$ which is seemingly new and interesting. Roughly speaking, our result implies that with any total mass there will be no overcrowding of cells provided that they initially distribute almost homogeneous spatially and the bounded domain is large enough, while the existing critical-mass theory is independent of the volume of the region.

The rest of this paper is organized as follows. In Section 2, we introduce some useful lemmas which are needed in the subsequent proof. In Section 3, we analyze the decay properties of the corresponding linearized systems in both $\gamma=0$ and $\gamma=1$. Delicate $L^p-L^q$ estimates are given which are crucial in our proof. In the last section, we prove our main results using the one-step contradiction argument borrowed from \cite{Win10, Cao} with slight modification.

\section{Preliminaries} 

First, we introduce the result on existence and uniqueness of local classical solutions to the Keller--Segel system \eqref{chemo1}. The proof for the doubly parabolic case $\gamma=1$ can be found in \cite{BBTW15, HW05} and the simplified case $\gamma=0$ can be done in almost the same way except when dealing with the second elliptic equation, classical elliptic theory together with Sobolev embeddings are needed. 
\begin{lemma}[Local well-posedness] \label{loex1}Assume $\rho_0\in C(\overline{\Omega})$ and $\gamma c_0\in W^{1,\sigma}(\Omega)$ are non-negative with $\sigma>d$. Then there exists $T_{\mathrm{max}}>0$ such that \eqref{chemo1} possesses a unique  classical solution $(\rho,c)\in \left(C([0,T_{\mathrm{max}})\times\overline{\Omega})\cap C^{2,1}(\overline{\Omega}\times (0,T_{\mathrm{max}}))\right)^2$ which is non-negative. Moreover, $T_{\mathrm{max}}<\infty$ if and only if
\begin{equation}
	\limsup\limits_{t\nearrow T_{\mathrm{max}}}\|\rho(\cdot,t)\|_{L^\infty(\Omega)}=\infty.
\end{equation}
	For any $\theta>\frac{d}{2}$, if the solution of \eqref{chemo1} satisfies
	\begin{equation}
		\|\rho(\cdot,t)\|_{L^\theta(\Omega)}<\infty\qquad\text{for all}\;t\in(0,T_{\mathrm{max}}),
	\end{equation}then $T_{\mathrm{max}}=\infty,$ and there holds
	\begin{equation}
		\sup\limits_{t>0}(\|\rho(\cdot,t)\|_{L^\infty(\Omega)}+\|c(\cdot,t)\|_{W^{1,\infty}(\Omega)})<\infty.
	\end{equation}
\end{lemma}

 The following result shows that a lower-order perturbation to a sectorial operator is still a sectorial operator \cite{Zheng04,Engel}.
 \begin{lemma}\label{pertur1}
 	Suppose that $A$ is a sectorial operator and $B$ is a linear operator with $D(A)\subset D(B)$ such that for any $x\in D(A)$, there holds
 	\begin{equation}
 		\|Bx\|\leq \varepsilon \|Ax\|+K_\varepsilon \|x\|
 	\end{equation}where $\varepsilon>0$ is an arbitrary small constant and $K_\varepsilon$ is a positive constant depending on $\varepsilon.$ Then $A+B$ is sectorial.
 \end{lemma}

The following lemma presents an estimate for frequently used integrals throughout this paper, the proof of which can be found in \cite{Win10}.
 \begin{lemma}\label{lmint}
 	Suppose $0<\alpha<1$, $0<\beta<1$, $\gamma>0$, $\delta>0$ and $\gamma\neq\delta$. Then there holds
 	\begin{equation}
 		\int_0^t(1+(t-s)^{-\alpha})e^{-\gamma(t-s)}(1+s^{-\beta})e^{-\delta s}ds\leq C(\alpha,\beta,\delta,\gamma)(1+t^{\min\{0,1-\alpha-\beta\}})e^{-\min\{\gamma,\delta\}t}
 	\end{equation}for all $t>0$, where $C(\alpha,\beta,\delta,\gamma)=C\cdot(\frac{1}{|\delta-\gamma|}+\frac{1}{1-\alpha}+\frac{1}{1-\beta})$.
 \end{lemma}

  Last, we recall the important $L^p-L^q$ estimates for the Neumann heat semigroup on bounded domains (see e.g., \cite{Cao,Win10}).
 \begin{lemma}\label{lmpq}
 	Suppose $\{e^{t\Delta}\}_{t\geq0}$ is the Neumann heat semigroup in $\Omega$, and $\lambda_1>0$ denote the first nonzero eigenvalue of $-\Delta$ in $\Omega$ under Neumann boundary conditions. Then there exist $k_1,..., k_4>0$ which only depend on $\Omega$ such that the following properties hold:
 	\begin{enumerate}[(i)]
 		\item If $1\leq q\leq p\leq \infty,$ then
 	\begin{equation}
 		\|e^{t\Delta}w\|_{L^p(\Omega)}\leq k_1(1+t^{-\frac{d}{2}(\frac1q-\frac1p)})e^{-\lambda_1t}\|w\|_{L^q(\Omega)}\qquad\text{for all}\;\;t>0
 	\end{equation}for all $w\in L^q_0(\Omega)$;
 	\item  If $1\leq q\leq p\leq \infty,$ then
 	\begin{equation}
 		\|\nabla e^{t\Delta}w\|_{L^p(\Omega)}\leq k_2(1+t^{-\frac12-\frac{d}{2}(\frac1q-\frac1p)})e^{-\lambda_1t}\|w\|_{L^q(\Omega)}\qquad\text{for all}\;\;t>0
 	\end{equation}for each $w\in L^q(\Omega)$;
 	\item If $2\leq q\leq p<\infty,$ then
 	\begin{equation}
 		\|\nabla e^{t\Delta}w\|_{L^p(\Omega)}\leq k_3e^{-\lambda_1 t}(1+t^{-\frac{d}{2}(\frac1q-\frac1p)})\|\nabla w\|_{L^q(\Omega)}\qquad\text{for all}\;\;t>0
 	\end{equation}for all $w\in W^{1,p}(\Omega)$;
 	\item  If $1<q\leq p\leq \infty,$ then
 	\begin{equation}
 		\|e^{t\Delta}\nabla \cdot w\|_{L^p(\Omega)}\leq k_4(1+t^{-\frac12-\frac{d}{2}(\frac1q-\frac1p)})e^{-\lambda_1 t}\|w\|_{L^q(\Omega)}\quad\text{for all}\;\;t>0
 	\end{equation}for any $w\in (W^{1,p}(\Omega))^d.$
 	\end{enumerate}
 	 	 \end{lemma}

\section{Decay Properties of the Linearized Systems}
In this section, we try to establish $L^p-L^q$ estimates for the associated semigroups of linearized system \eqref{chemo2l}.

\subsection{The parabolic--elliptic case: $\gamma=0$}

In this part, we consider the simplified parabolic--elliptic case $\gamma=0$. Now, the linearized problem reads
\begin{equation}
\begin{cases}\label{chemo4}
		\u_t-\Delta \u+\M\Delta \v=0 \qquad &x\in\Omega,\;t>0\\
	-\Delta \v+\v=\u \qquad &x\in\Omega,\;t>0\\
	\frac{\partial \u}{\partial\nu}=\frac{\partial \v}{\partial\nu}=0,\qquad &x\in\partial\Omega,\;t>0\\
		\u(x,0)=u_0(x)\qquad & x\in\Omega
\end{cases}
\end{equation}or equivalently,
\begin{equation}\begin{cases}
				\u_t-\Delta \u+\M\Delta (I-\Delta)^{-1}\u=0 \qquad &x\in\Omega,\;t>0\\
	\frac{\partial \u}{\partial\nu}=0,\qquad &x\in\partial\Omega,\;t>0\\
		\u(x,0)=u_0(x)\qquad & x\in\Omega
\end{cases}
\end{equation}since $I-\Delta$ is invertible on $L^p_0(\Omega)$ for any $p>1.$

Denote $\mathcal{L}=\Delta-\M\Delta (I-\Delta)^{-1}$ with domain $D(\mathcal{L})=W^{2,p}_N(\Omega)\cap L^p_0(\Omega)$, where 
\begin{equation}
	W^{2,p}_N(\Omega):=\{w\in W^{2,p}(\Omega),\;\frac{\partial w}{\partial\nu}=0\;\text{on}\;\partial\Omega.\}\nonumber
\end{equation} Denoting $B=\Delta(I-\Delta)^{-1}$, then one easily checks that $B$ is a bounded operator on $L^p_0(\Omega).$ Indeed,  since $\Delta$ generates a $C_0-$semigroup of contraction on $L^p_0(\Omega)$, we infer by the Hill-Yosida theorem that $\|(I-\Delta)^{-1}\|_{L^p}\leq 1$. It follows that 
\begin{equation}
	\|\Delta(I-\Delta)^{-1}\|_{L^p}=\|I-(I-\Delta)^{-1}\|_{L^p}\leq 2.
\end{equation}
Then by Lemma \ref{pertur1}, $\mathcal{L}$ is also a sectorial operator and thus generates an analytic semigroup on $L^p_0(\Omega)$, which is denoted by $e^{t\mathcal{L}}$ in the sequel. 
 
 Next, we show that the following exponentially decay estimates of $e^{t\mathcal{L}}$ in $L^2_0(\Omega).$

 \begin{lemma}\label{exdlem0}
 	Suppose $0< \M<1+\lambda_1$. Then for any $u_0\in L^2_0(\Omega),$ there holds
 	\begin{equation}\label{expdecay0}
 		\|e^{t\mathcal{L}}u_0\|\leq e^{-\mu_0t}\|u_0\|,\qquad\forall t\geq0
 	\end{equation} where $\mu_0\triangleq\lambda_1(1-\frac{\M}{1+\lambda_1}).$
 \end{lemma}
 \begin{proof}We only perform the formal energy estimates here which could be easily justified by density argument. First, we note that 
\begin{equation}\int_\Omega \u dx=\int_\Omega \v dx=0.\non
\end{equation}
Now, multiplying the second equation of \eqref{chemo4} by $-\Delta \v$ and integrating by parts, we get
 \begin{align}
 	\|\Delta \v\|^2+\|\nabla \v\|^2=\int_\Omega\nabla\u\cdot\nabla \v dx\leq \|\nabla \u\|\|\nabla\v\|.
 \end{align} Then observing  that by Poincar\'e's inequality \eqref{poin},
 \begin{equation}
 	\|\Delta\v\|^2\geq\lambda_1\|\nabla \v\|^2,
 \end{equation}we  infer that 
 \begin{equation}
 	\|\nabla\v\|\leq\frac{1}{1+\lambda_1}\|\nabla \u\|
 \end{equation}and hence
 	 \begin{equation}
 	\left|\int_\Omega\Delta \v \u dx\right|\leq \frac{1}{1+\lambda_1}\|\nabla \u\|^2.
 \end{equation}	
On the other hand, a multiplication of  the first equation by $\u$ and an integration over $\Omega$ yields that
 \begin{align}
 	\frac{1}{2}\frac{d}{dt}\|\u\|^2+\|\nabla \u\|^2=&-\M\int_\Omega\Delta \v  \u dx\nonumber\\
 	\leq& \frac{\M}{1+\lambda_1}\|\nabla \u\|^2.
 \end{align}	
 	Therefore, by Poincar\'e's inequality again, we have
 	\begin{equation}
 		\frac12\frac{d}{dt}\|\u\|^2+\lambda_1(1-\frac{\M}{1+\lambda_1})\|\u\|^2\leq 0
 	\end{equation} which yields
 	\begin{equation}\label{exde1}
 		\|\u(t)\|^2\leq \|u_{0}\|^2e^{-2\mu_0t},\quad\forall t\geq0.
 	\end{equation}with $\mu_0=\lambda_1(1-\frac{\M}{1+\lambda_1})$. This completes the proof. \end{proof}

%%%%%%%%%%%%%%%%%%%%%%%%%%New Proof%%%%%%%%%%%%%%%%%
Thanks to the exponentially decay property in $L^2_0(\Omega)$, we are able to prove the following
\begin{lemma}\label{lmdecay}
	Suppose $0<\M<1+\lambda_1$. For any $p>1$, $1\leq q\leq p\leq\infty$  and $u_0\in L^p_0(\Omega)$, there holds
\begin{equation}\label{decayest}
	\|e^{t\mathcal{L}}u_0\|_{L^p(\Omega)}\leq c_1(1+t^{-\frac{d}{2}(\frac1q-\frac1p)})e^{-\mu_0t}\|u_0\|_{L^q(\Omega)},\qquad\forall t>0 
\end{equation}where $c_1>0$ depends only on $d$ and $\Omega$.
\end{lemma}

\begin{proof} The proof consists of two parts. First, we derive for $t\leq1$, there holds
\begin{equation}\nonumber
	\|e^{t\mathcal{L}}u_0\|_{L^p(\Omega)}\leq ct^{-\frac{d}{2}(\frac1q-\frac1p)})\|u_0\|_{L^q(\Omega)}. 
\end{equation} Then  for $t\geq1$, arguing in the same way as done in \cite{Win10}, invoking Lemma \ref{exdlem0}, we prove that
\begin{equation}\nonumber
		\|e^{t\mathcal{L}}u_0\|_{L^p(\Omega)}\leq ce^{-\mu_0t}\|u_0\|_{L^q(\Omega)}.
\end{equation}Then our assertion follows by combining the above two relations.

 By the variation-of-constants formula, we observe
\begin{equation}
	\u(t)=e^{t\Delta}u_0-\M\int_0^te^{(t-s)\Delta}\Delta \v(s)ds.
\end{equation}
Denoting $\u(t)=e^{t\mathcal{L}}u_0$, it follows from above and Lemma \ref{lmpq} that for $t\leq1,$
%%%%%%%%%Gronwall inequality%%%%%%%%%%%%%

	\begin{align}
	\|e^{t\mathcal{L}}u_0\|_{L^p(\Omega)}=&\bigg{\|}e^{t\Delta}u_0-\M\int_0^t e^{(t-s)\Delta}\Delta \v(s)ds\bigg{\|}_{L^p(\Omega)}\nonumber\\
\leq&\|e^{t\Delta}u_0\|_{L^p(\Omega)}+k_1\M\int_0^t e^{-\lambda_1(t-s)} \|\Delta \v(s)\|_{L^p(\Omega)}ds\non\\
\leq &k_1t^{-\frac{d}{2}(\frac{1}{q}-\frac{1}{p})}\|u_0\|_{L^q(\Omega)}+ 2k_1\M\int_0^t   e^{-\lambda_1(t-s)} \|\u(s)\|_{L^p(\Omega)}ds
	\end{align}since $\Delta \v=\Delta(I-\Delta)^{-1}\u$ and $\|\Delta(I-\Delta)^{-1}\|_{L^p(\Omega)}\leq2$ for any $1<p<\infty$ which also holds true for $p=\infty$. Indeed, noting that $\|\v\|_{L^\infty(\Omega)}\leq \|\u\|_{L^\infty(\Omega)}$ by standard energy estimates, we infer that $\|\Delta \v\|_{L^\infty(\Omega)}\leq \|\u\|_{L^\infty(\Omega)}+\|\v\|_{L^\infty(\Omega)}\leq 2\|\u\|_{L^\infty(\Omega)}$.	
	 
Let $y(t)=t^{\frac{d}{2}(\frac1q-\frac1p)}\|\u(t)\|_{L^p(\Omega)}$. Then we derive that for $t\leq1$
\begin{equation}
	y(t)\leq k_1\|u_0\|_{L^q(\Omega)}+2k_1\M\int_0^te^{-\lambda_1(t-s)}s^{-\frac{d}{2}(\frac1q-\frac1p)}y(s)ds.
\end{equation}An application of Gronwall's inequality yields that
\begin{align}
		y(t)\leq& k_1\|u_0\|_{L^q(\Omega)}\exp\{2k_1\M\int_0^te^{-\lambda_1(t-s)}s^{-\frac{d}{2}(\frac1q-\frac1p)}ds\}\nonumber\\
		\leq&k_1\|u_0\|_{L^q(\Omega)}\exp\{\frac{2k_1\M}{1-\frac{d}{2}(\frac1q-\frac1p)} t^{1-\frac{d}{2}(\frac1q-\frac1p)}\}\nonumber\\
		\leq &k_1\|u_0\|_{L^q(\Omega)}e^{4k_1\M}\end{align}provided that $\frac1d>\frac1q-\frac1p$.
Therefore, for $t\leq1$ and $\frac1d>\frac1q-\frac1p$,
\begin{equation}
	\|e^{t\mathcal{L}}u_0\|_{L^p(\Omega)}\leq c_2t^{-\frac{d}{2}(\frac{1}{q}-\frac{1}{p})}\|u_0\|_{L^q(\Omega)}
\end{equation}where $c_2=k_1e^{4k_1\M}<k_1e^{4k_1(1+\lambda_1)}.$ 

Similarly, if $\frac1q-\frac1p<\frac{N}{d}$ for some $2\leq N\in\mathbb{N}$, we may find $\{q_k\}_{k=0}^{N-1}$ between $q=q_N$ and $p=q_0$ such that $\frac1d>\frac1{q_{k+1}}-\frac1{q_{k}}$ and hence
\begin{align}
	\|e^{t\mathcal{L}}u_0\|_{L^p(\Omega)}\leq &c_2(\frac{t}{N})^{-\frac{d}{2}(\frac1{q_1}-\frac1p)}\|e^{\frac{N-1}{N}t\mathcal{L}}u_0\|_{L^{q_1}(\Omega)}\nonumber\\
	\leq& c_2^2(\frac{t}{N})^{-\frac{d}{2}(\frac1{q_1}-\frac1p)}(\frac{t}{N})^{-\frac{d}{2}(\frac1{q_2}-\frac1{q_1})}\|e^{\frac{N-2}{N}t\mathcal{L}}u_0\|_{L^{q_2}(\Omega)}\nonumber\\
	\leq&...\nonumber\\
	\leq& c_2^N(\frac{t}{N})^{-\frac{d}{2}(\frac1{q_1}-\frac1p)}...(\frac{t}{N})^{-\frac{d}{2}(\frac1{q}-\frac1{q_{N-1}})}\|u_0\|_{L^{q}(\Omega)}\nonumber\\
	= &c_2^N (\frac{t}{N})^{-\frac{d}{2}(\frac1{q}-\frac1p)}\|u_0\|_{L^q(\Omega)}\nonumber\\
	\leq& c_3t^{-\frac{d}{2}(\frac1{q}-\frac1p)}\|u_0\|_{L^q(\Omega)}\label{pq0}
\end{align}
where $c_3=(c_2\sqrt{N})^N$ since $\frac{d}{2}(\frac{1}{q}-\frac{1}{p})<\frac N2$. 

Obviously, $N\leq d+1$. Thus, there is $c_4>0$ depends on $d$ and $\Omega$ only such that for all $p>1$ and $1\leq q\leq p\leq \infty$, there holds for $t\leq1$
\begin{equation}\label{pq0a}
		\|e^{t\mathcal{L}}u_0\|_{L^p(\Omega)}\leq c_4t^{-\frac{d}{2}(\frac1{q}-\frac1p)}\|u_0\|_{L^{q}(\Omega)}.
\end{equation}

Now, for $t\geq1$ and $p\geq2$, we derive by \eqref{pq0a} and Lemma \ref{exdlem0} that
\begin{align}
	\|e^{t\mathcal{L}}u_0\|_{L^p(\Omega)}\leq& c_4 2^{\frac d2(\frac12-\frac1p)}\|e^{(t-\frac12)\mathcal{L}}u_0\|_{L^2(\Omega)}\non\\
	\leq& c_42^{\frac d2(\frac12-\frac1p)}e^{-\mu_0(t-1)}\|e^{\frac12\L}u_0\|_{L^2(\Omega)}\non\\
	\leq & c_4^22^{\frac d2(\frac12-\frac1p)}e^{-\mu_0(t-1)}\times\max\{2^{\frac d2(\frac12-\frac1p)},|\Omega|^{\frac12-\frac1q}\}\|u_0\|_{L^q(\Omega)}\label{pq1}\end{align}
	and for $p<2,$ we have
\begin{align}
	\|e^{t\L}u_0\|_{L^p(\Omega)}\leq& |\Omega|^{\frac1p-\frac12}\|e^{t\L}u_0\|_{L^2(\Omega)}\nonumber\\
	\leq&|\Omega|^{\frac1p-\frac12}e^{-\mu_0(t-\frac12)}\|e^{\frac12\L}u_0\|_{L^2(\Omega)}\nonumber\\
	\leq&c_4 2^{\frac d2(\frac12-\frac1p)}|\Omega|^{\frac1p-\frac12}e^{-\mu_0(t-\frac12)}\|u_0\|_{L^q(\Omega)}\label{pq2}.
\end{align}	
Finally, a combination of 	\eqref{pq0a}, \eqref{pq1} and \eqref{pq2} completes the proof.\end{proof}

%%%%%%%%%%%%%%%%%%%%%%%%%%%%%%
With minor modification and in the same manner as done in proof of Lemma \ref{lmdecay}, we can prove the following result for the special case when $u_0=\nabla \cdot w$.
\begin{lemma}\label{cor2}Suppose $0<\M<1+\lambda_1$. For any $1< q\leq p\leq\infty$  and $u_0=\nabla \cdot w$, there holds
\begin{equation}\label{decayest0}
	\|e^{t\mathcal{L}}u_0\|_{L^p(\Omega)}\leq c_5(1+t^{-\frac12-\frac{d}{2}(\frac1q-\frac1p)})e^{-\mu_0t}\|w\|_{L^q(\Omega)} 
\end{equation}where $c_5$ depends only on $d$ and $\Omega$.

\end{lemma}

\subsection{The fully parabolic case: $\gamma=1$}
In this part, we consider the case $\gamma=1$. Similar as before, denote $(\tilde{u},\tilde{v})$ the solution to corresponding linearized system to \eqref{chemo2}. Then, $(\tilde{u},\tilde{v})$ satisfies
\begin{equation}
	\begin{cases}\label{chemo3b}
		\u_t-\Delta \u+\M\Delta \v=0,\qquad &x\in\Omega,\;t>0\\
		\v_t-\Delta \v+\v= \u,\qquad &x\in\Omega,\;t>0\\
		\frac{\partial \u}{\partial\nu}=\frac{\partial \v}{\partial\nu}=0,\qquad &x\in\partial\Omega,\;t>0\\
		\u(x,0)=u_0(x),\;\;\v(x,0)= v_0(x)\qquad & x\in\Omega.
	\end{cases}
\end{equation}
%%%%%%%%%%%%%%%%%%%%%%%%%%%%%%%%%%%%%%%%%%%
Denote $\Delta$ the usual Laplacian operator with homogeneous Neumann boundary condition. Since $-\Delta$ is analytic on $L^p_0(\Omega)$ with domain $D_p(\Delta)=W_N^{2,p}(\Omega)\cap L^p_0(\Omega)$,  we can define the power $(-\Delta)^s$ of $-\Delta$ for any $s\in\mathbb{R}$ and we denote the domain of $(-\Delta)^s$ in $L^p_0(\Omega)$ by $D_p((-\Delta)^s).$

Let $\mathcal{X}=L^p_0(\Omega)\times D_p((-\Delta)^\frac12)$ for $1< p<\infty$ with norm
\begin{equation}
	\|(u,v)\|_{\mathcal{X}}=\|u\|_{L^p(\Omega)}+\|\nabla v\|_{L^p(\Omega)}
\end{equation}
and define
\begin{align}
\mathcal{A}=\left(\begin{matrix}\Delta &-\M\Delta\\
1&\Delta -1\end{matrix}\right)
\end{align}with domain $D(\mathcal{A})=D_p(\Delta)\times  D_p((-\Delta)^\frac32)$.
We observe that
\begin{align}
	\mathcal{A}=\left(\begin{matrix}\Delta &0\\
0&\Delta -1\end{matrix}\right)+\left(\begin{matrix}0 &-\M\Delta\\
1&0\end{matrix}\right)
\triangleq \Lambda+\mathcal{U}
\end{align}
where $\Lambda$ is a sectorial operator on $\mathcal{X}$. Moreover, one easily verifies that 
 $D(\mathcal{A})=D(\Lambda)\subset D(\mathcal{U})=  D_p((-\Delta)^\frac12)\times D_p(\Delta)$ and 
for any $(u,v)\in D(\Lambda)$, by interpolation, there holds
\begin{equation}
	\M\|\Delta v\|_{L^p}+\|\nabla u\|_{L^p}\leq \varepsilon\bigg(\|\Delta u\|_{L^p}+\|\nabla \Delta v-\nabla v\|_{L^p}\bigg)+K_{\varepsilon}(\|u\|_{L^p}+\|\nabla v\|_{L^p}).
\end{equation} Then Lemma \ref{pertur1} indicates that $\mathcal{A}$ is a sectorial operator as well. For the sake of convenience, we denote 
\begin{align}
	\left(\begin{matrix}\u(t)\\
\v(t)\end{matrix}\right)=e^{t\mathcal{A}}\left(\begin{matrix}u_0\\
v_0\end{matrix}\right)\triangleq\left(\begin{matrix}\Phi_1^t(u_0,v_0)\\
\Phi_2^t(u_0,v_0)\end{matrix}\right).
\end{align}

Now, similar as before, we first prove the exponentially decay property for the semigroup $e^{t\mathcal{A}}$ in the Hilbert space $L^2_0\times (H^1\cap  L^2_0).$

\begin{lemma}\label{lmexdecaypp}
	Assume $0<\M<1+\lambda_1$. For any given initial data $u_0\in L^2_0(\Omega)$ and $v_0\in H^1(\Omega)\cap L^2_0(\Omega)$, the solution of \eqref{chemo3b} satisfies the following exponentially decay estimate
	\begin{equation}
		\|\u\|^2+\M\|\nabla \v\|^2\leq e^{-2\mu_1t}(\|u_0\|^2+\M\|\nabla v_0\|^2)\qquad\text{for all}\;t\geq0
	\end{equation}where $\mu_1\triangleq\lambda_1-\frac12\bigg(\sqrt{4\lambda_1\M+1}-1\bigg)>0.$
\end{lemma}
\begin{proof}Here we only perform formal energy estimates which could be rigorously justified by density arguments. First, we observe that 
\begin{equation}
	\int_\Omega \u(t)dx=\int_\Omega \v(t)dx=0.\nonumber
\end{equation}
Multiplying the first equation by $\u$ and the second equation by $-\M\Delta \v$, integrating by parts and adding the resultant up, we obtain that
	\begin{align}
		\frac{1}{2}\frac{d}{dt}\bigg(\|\u\|^2+\M\|\nabla \v\|^2\bigg)+\M\bigg(\|\nabla \v\|^2+\|\Delta \v\|^2\bigg)+\|\nabla \u\|^2=&2\M\int_\Omega\nabla \u\cdot \nabla \v dx\nonumber\\
		\leq&\delta\|\nabla \u\|^2+\frac{\M^2}{\delta}\|\nabla \v\|^2.
	\end{align}
In views of Poincar\'e's lemma, we infer that
\begin{align}
\frac{1}{2}\frac{d}{dt}\bigg(\|\u\|^2+\M\|\nabla \v\|^2\bigg)+\M(\lambda_1+1-\frac{\M}{\delta})\|\nabla \v\|^2+\lambda_1(1-\delta)\| \u\|^2\leq0.
\end{align}	
	Picking $\delta=\delta_0\triangleq\frac{1}{2}\bigg(\frac{\sqrt{4\lambda_1\M+1}}{\lambda_1}-\frac{1}{\lambda_1}\bigg)$, we get 
	\begin{equation}
		\lambda_1+1-\frac{\M}{\delta}=\lambda_1(1-\delta)\equiv\mu_1>0
	\end{equation}whenever $\M<1+\lambda_1$ holds. Here, $\delta$ is picked such that  $\lambda_1+1-\frac{\M}{\delta}=\lambda_1(1-\delta)$ holds and attains the maximum at $\delta_0$.  Then the conclusion follows from solving an ordinary differential inequality.	\end{proof}

The next $L^p-L^q$ estimates for $e^{t\mathcal{A}}$ plays a key role in the proof for the case $\gamma=1$, which is established based on Lemma \ref{lmpq} and Lemma \ref{lmexdecaypp} by similar arguments as we done in the previous part. However,  a coupling (linear) system is now under consideration. The  presence of the cross diffusion and  the restrictions on the parameter $q$ in Lemma \ref{lmpq}-(iii,iv) bring a lot of difficulties and hence the calculations here are more involved.
 \begin{lemma}\label{keylem1} Assume $d\geq2$ and $0<\M<1+\lambda_1$. Then for any $u_0\in C(\overline{\Omega})\cap L^1_0(\Omega)$ and $v_0\in C^1(\overline{\Omega})\cap L^1_0(\Omega)$ satisfying  $\partial_\nu v_0=0$ on $\partial\Omega$,  there holds
	\begin{equation}
		\|\tilde{u}(t)\|_{L^p(\Omega)}\leq c_6e^{-\mu_1t}(1+t^{-\frac{d}{2}(\frac2d-\frac1p)})\big(\|u_0\|_{L^{d/2}(\Omega)}+\|\nabla v_0\|_{L^d(\Omega)}\big)\qquad\forall t>0
	\end{equation} for any $p>1$ satisfying  $\frac d2\leq p<\infty$,  and
	\begin{equation}
		\|\nabla \tilde{v}(t)\|_{L^p(\Omega)}\leq c_7pe^{-\mu_1t}(1+t^{-\frac{d}{2}(\frac1d-\frac1p)})\big(\|u_0\|_{L^{d/2}(\Omega)}+\|\nabla v_0\|_{L^d(\Omega)}\big)\qquad\forall t>0
	\end{equation}for any $p$ satisfying $d\leq p< \infty$, where $c_6,c_7>0$ depend on $d$, $\M$ and $\Omega$ only.
\end{lemma}
\begin{proof}The proof  consists of several steps.

%%%%%%%%%%%%%%%%%%%%%%%%%%%%%%%%%%%%%%%%%%%%%%%%%%%%%%%%%%%%%%
%%%%%%%%%%%%Step 1
%%%%%%%%%%%%%%s%%%%%%%%%%%%%%%%%%%%%%%%%%%%%%%%%%%%%%%%%%%%%%
\textbf{Step 1.} 
We derive from variation-of-constants formula  the following expressions for solutions of \eqref{chemo3b}  for all $t>0$\begin{equation}
		\u(t)=e^{t\Delta}u_0-\M\int_0^te^{(t-s)\Delta}\Delta \v(s)ds,\nonumber
\end{equation}and
\begin{equation}
	\v(t)=e^{t(\Delta-1)}v_0+\int_0^te^{(t-s)(\Delta-1)}\u(s)ds.\nonumber
\end{equation}
Substituting $\v$ into the expression of $\u$ leads to \begin{align}
	\u(t)=&e^{t\Delta}u_0-\M\int_0^te^{(t-s)\Delta}\Delta e^{s(\Delta-1)} v_0ds-\M\int_0^te^{(t-s)\Delta}\Delta \int_0^s e^{(s-\tau)(\Delta-1)}\u(\tau)d\tau ds.\non\end{align} %%%%%%%%%%%%%%%%%%%%%%%%%%%%%%%%%%%%%%%%%%%%%%%%%%%%%%%%%%
%%%%%%%%%%%%%%%%%%%%%%%%%%%%%%%%%%%%%%%%%%%%%%%%%%%%%%%%%%%%%%
%%%%%%%%%%%%Step 2
%%%%%%%%%%%%%%s%%%%%%%%%%%%%%%%%%%%%%%%%%%%%%%%%%%%%%%%%%%%%%

\textbf{Step 2.}  We claim that when $t\leq1$, for any $1<p<\infty$ and any $r\geq2$ satisfying $\frac1r-\frac1p<\frac1d$, there holds
\begin{equation}\label{claim0}
	\|\int_0^te^{(t-s)\Delta}\Delta e^{s(\Delta-1)} v_0ds\|_{L^p(\Omega)}\leq c\|\nabla v_0\|_{L^r(\Omega)}
\end{equation} with $c$ depends on $\Omega$ only.

In fact, if $2\leq r\leq p<\infty$, by Lemma \ref{lmpq} and Lemma \ref{lmint}, we infer that
\begin{align}
	&\|\int_0^te^{(t-s)\Delta}\Delta e^{s(\Delta-1)} v_0ds\|_{L^p(\Omega)}\non\\
	\leq&k_4\int_0^te^{-\lambda_1(t-s)}(1+(t-s)^{-\frac12})\|\nabla e^{s(\Delta-1)}v_0\|_{L^p(\Omega)}ds\nonumber\\
	\leq& k_3k_4\int_0^t e^{-\lambda_1(t-s)}(1+(t-s)^{-\frac12}) e^{-s(\lambda_1+1)}(1+s^{-\frac{d}{2}(\frac1r-\frac1p)})\|\nabla v_0\|_{L^r(\Omega)}ds\nonumber\\
	\leq& c\|\nabla v_0\|_{L^r(\Omega)}/(1-\frac d2(\frac1r-\frac1p) )\nonumber\\
	\leq & c\|\nabla v_0\|_{L^r(\Omega)}
\end{align}where $c$ depends on  $\Omega$ only  provided that $0\leq \frac1r-\frac1p<\frac{1}{d}$ (when $r=p$, we calculate directly without using Lemma \ref{lmint}). On the other hand, if $p<r$, applying H\"older's inequality, we have
\begin{align}
	&\|\int_0^te^{(t-s)\Delta}\Delta e^{s(\Delta-1)} v_0ds\|_{L^p(\Omega)}\non\\
	\leq&|\Omega|^{\frac1p-\frac1r}\int_0^t \|e^{(t-s)\Delta}\Delta e^{s(\Delta-1)} v_0\|_{L^r(\Omega)}ds\nonumber\\
	\leq&k_3k_4|\Omega|^{\frac1p-\frac1r}\int_0^t e^{-\lambda_1(t-s)}(1+(t-s)^{-\frac12}) e^{-s(\lambda_1+1)} \|\nabla v_0\|_{L^r(\Omega)}ds\nonumber\\
	\leq & c\|\nabla v_0\|_{L^r(\Omega)}.
\end{align}

%%%%%%%%%%%%%%%%%%%%%%%%%%%%%%%%%%%%%%%%%%%%%%%%%%%%%%%%%%%%%%%%%%%%%%%%%%%%%%%%%%%%%%%%%%%%%%%%%%%%%%%%%%%%%%%%%%%%%%
%%%%%%%%%%%%Step 3
%%%%%%%%%%%%%%s%%%%%%%%%%%%%%%%%%%%%%%%%%%%%%%%%%%%%%%%%%%%%%
\textbf{Step 3.} Now, due to Lemma \ref{lmpq}, Lemma \ref{lmint} and \eqref{claim0},  we infer for any $p>1$, $1\leq q\leq p<\infty$, $r\geq2$  and $t\leq 1$ that
\begin{align}\label{tuexp}
	&\|\u(t)\|_{L^p(\Omega)}
	\leq \|e^{t\Delta}u_0\|_{L^p(\Omega)}+\M\int_0^t\|e^{(t-s)\Delta}\Delta e^{s(\Delta-1)}v_0\|_{L^p(\Omega)}ds\non\\
	&\quad+\M\int_0^t\bigg{\|}e^{(t-s)\Delta}\nabla\cdot\left(\nabla\int_0^s   e^{(s-\tau)(\Delta-1)}\u(\tau)d\tau\right)\bigg{\|}_{L^p(\Omega)}ds\nonumber\\
	\leq& k_1 (1+t^{-\frac{d}{2}(\frac1q-\frac1p)})\|u_0\|_{L^{q}(\Omega)}+c \|\nabla v_0\|_{L^r(\Omega)}\non\\
	+&k_2k_4\M\int_0^te^{-\lambda_1(t-s)}(1+(t-s)^{-\frac12})\int_0^se^{-(\lambda_1+1)(s-\tau)}(1+(s-\tau)^{-\frac12})\|\u(\tau)\|_{L^p(\Omega)}d\tau ds\nonumber\\
	\leq &2k_1 t^{-\frac{d}{2}(\frac1q-\frac1p)}\|u_0\|_{L^{q}(\Omega)}+c\|\nabla v_0\|_{L^r(\Omega)}\nonumber\\
	+&k_2k_4\M\int_0^t\int_0^se^{-\lambda_1(t-s)}(1+(t-s)^{-\frac12})e^{-(\lambda_1+1)(s-\tau)}(1+(s-\tau)^{-\frac12})\|\u(\tau)\|_{L^p(\Omega)}d\tau ds
\end{align}provided that $\frac{1}{r}-\frac{1}{p}<\frac{1}{d}$.

Letting $y(t)=t^{\frac{d}{2}(\frac1q-\frac1p)}\|\u(t)\|_{L^p(\Omega)}$, we derive by changing the order in integrations  that for $t\leq1$,
\begin{align}
	y(t)\leq &c\bigg( \|u_0\|_{L^q(\Omega)}+c\|\nabla v_0\|_{L^r(\Omega)}\bigg)\non\\
	+&c\M\int_0^t\int_0^se^{-\lambda_1(t-s)}e^{-(\lambda_1+1)(s-\tau)}(1+(t-s)^{-\frac12})(1+(s-\tau)^{-\frac12})\tau^{-\frac{d}{2}(\frac1q-\frac1p)}y(\tau)d\tau ds\nonumber\\
	=& c\bigg( \|u_0\|_{L^q(\Omega)}+\|\nabla v_0\|_{L^r(\Omega)}\bigg)\non\\
	+&c\M\int_0^t\bigg(\int_\tau^te^{-\lambda_1(t-s)}e^{-(\lambda_1+1)(s-\tau)}(1+(t-s)^{-\frac12})(1+(s-\tau)^{-\frac12})ds\bigg)\tau^{-\frac{d}{2}(\frac1q-\frac1p)}y(\tau)d\tau.
\end{align}By Lemma \ref{lmint} and  direct calculations, we have for $0\leq\frac1q-\frac1p<\frac2{d}$ and for $t\leq1$,
\begin{align}
	&\int_0^t\bigg(\int_\tau^te^{-\lambda_1(t-s)}e^{-(\lambda_1+1)(s-\tau)}(1+(t-s)^{-\frac12})(1+(s-\tau)^{-\frac12})ds\bigg)\tau^{-\frac{d}{2}(\frac1q-\frac1p)}d\tau\nonumber\\
	\leq &c\int_0^t\tau^{-\frac{d}{2}(\frac1q-\frac1p)}e^{-\lambda_1(t-\tau)}d\tau\nonumber\\
	\leq&c\int_0^1\tau^{-\frac{d}{2}(\frac1q-\frac1p)}d\tau\nonumber\\
	\leq& \frac{c}{1-\frac d2(\frac1q-\frac1p)}
\end{align}with $c$ depends on  $\Omega$ only.

%%%%%%%%%%%%%%%%%%%%%%%%%%%%%%%%%%%%%%%%%%%%%%%%%%%%%%%%%%%%%%
%%%%%%%%%%%%Step 4
%%%%%%%%%%%%%%s%%%%%%%%%%%%%%%%%%%%%%%%%%%%%%%%%%%%%%%%%%%%%%
\textbf{Step 4.} Choosing $r=q=l\geq2$ in \eqref{tuexp}, then for any $l\leq p<\infty$ satisfying $0\leq\frac1l-\frac1p<\frac1{2d}$, we deduce by Gronwall's inequality that for $t\leq1$,
\begin{equation}
	\|\u(t)\|_{L^p(\Omega)}\leq ct^{-\frac d2(\frac1l-\frac1p)}(\|u_0\|_{L^{l}(\Omega)}+\|\nabla v_0\|_{L^l(\Omega)})
\end{equation}with $c$ depends on $\Omega$ only. As a consequence, we deduce for $t\leq1$  and $2\leq l\leq p<\infty$ satisfying $\frac1l-\frac1p<\frac1{2d}$ that
\begin{align}
		\|\nabla \v(t)\|_{L^p(\Omega)}\leq&\|\nabla e^{t(\Delta-1)}v_0\|_{L^p(\Omega)}+\int_0^t\|\nabla e^{(t-s)(\Delta-1)}\u(s)\|_{L^p(\Omega)}ds\nonumber\\
	\leq& 2k_3t^{-\frac{d}{2}(\frac1l-\frac1p)}\|\nabla v_0\|_{L^l(\Omega)}+k_2\int_0^te^{-(\lambda_1+1)(t-s)}(1+(t-s)^{-\frac12})\|\u(s)\|_{L^p(\Omega)}ds\nonumber\\
	\leq& 2k_3t^{-\frac{d}{2}(\frac1l-\frac1p)}\|\nabla v_0\|_{L^l(\Omega)}\non\\
	&+ck_2 ( \|u_0\|_{L^{l}}+\|\nabla v_0\|_{L^l(\Omega)})\int_0^te^{-(\lambda_1+1)(t-s)}(1+(t-s)^{-\frac12})s^{-\frac{d}{2}(\frac1l-\frac1p)}ds\nonumber\\
	\leq& ct^{-\frac{d}{2}(\frac1l-\frac1p)}(\|u_0\|_{L^{l}}+\|\nabla v_0\|_{L^l(\Omega)})
\end{align}with $c$ depends on  $\Omega$ only, since by Lemma \ref{lmint} again, when $t\leq1$ and  $0\leq\frac1l-\frac1p<\frac1{2d}$,
\begin{align}
	&\int_0^te^{-(\lambda_1+1)(t-s)}(1+(t-s)^{-\frac12})s^{-\frac{d}{2}(\frac1l-\frac1p)}ds\nonumber\\
	\leq& e^{\lambda_1}\int_0^te^{-(\lambda_1+1)(t-s)}(1+(t-s)^{-\frac12})s^{-\frac{d}{2}(\frac1l-\frac1p)}e^{-\lambda_1s}ds\nonumber\\
	\leq &c(1+t^{\min\{\frac12-\frac d2(\frac1l-\frac1p),0\}})\non\\
	\leq&c\nonumber\\
	\leq& ct^{-\frac{d}{2}(\frac1l-\frac1p)}.
\end{align}

 Summing up, we have for $t\leq1$ and any $2\leq l\leq p< \infty$ satisfying $0\leq \frac1l-\frac1p<\frac1{2d}$ that \begin{equation}\label{estpl0}
	\|\u(t)\|_{L^p(\Omega)}+\|\nabla\v(t)\|_{L^p(\Omega)}\leq ct^{-\frac d2(\frac1l-\frac1p)}(\|u_0\|_{L^{l}(\Omega)}+\|\nabla v_0\|_{L^l(\Omega)})
\end{equation}with $c$ depends on $\Omega$ only.

On the other hand, for  any $2\leq l\leq p<\infty$ such that $\frac1l-\frac1p\geq\frac1{2d}$, we may split $(\frac1p,\frac1l)$ and  $(0,t)$ evenly into $N-$intervals, respectively, with $N=d+1$. Denoting the end-points for $N-$intervals of $(\frac1p,\frac1l)$ by $\frac1p=\frac1{l_0}<\frac1{l_1}<...<\frac1{l_N}=\frac1l$, then $\frac1{l_{k+1}}-\frac1{l_k}<\frac1{2d}$ and by iteration, we deduce that  for $t\leq1$
\begin{align}\label{estpl}
	&\|\u(t)\|_{L^p(\Omega)}+\|\nabla \v(t)\|_{L^p(\Omega)}\non\\
	\leq & c(\frac{t}{N})^{-\frac{d}{2}(\frac1{l_1}-\frac1{p})}\left(\|\u(\frac{N-1}{N}t)\|_{L^{l_{1}}(\Omega)}+\|\nabla\v(\frac{N-1}{N}t)\|_{L^{l_{1}}(\Omega)}\right)\nonumber\\
	\leq&...\non\\
	 \leq& c^N(\frac{t}{N})^{-\frac{d}{2}(\frac1l-\frac1p)}(\|u_0\|_{L^{l}(\Omega)}+\|\nabla v_0\|_{L^l(\Omega)})\nonumber\\
	\leq & ct^{-\frac{d}{2}(\frac1l-\frac1p)}(\|u_0\|_{L^{l}(\Omega)}+\|\nabla v_0\|_{L^l(\Omega)})
\end{align}with $c$ depends on $d$ and $\Omega$ only.

%%%%%%%%%%%%%%%%%%%%%%%%%%%%%%%%%%%%%%%%%%%%%%%%%%%%%%%%%%%%%%
%%%%%%%%%%%%Step 5
%%%%%%%%%%%%%%s%%%%%%%%%%%%%%%%%%%%%%%%%%%%%%%%%%%%%%%%%%%%%%

\textbf{Step 5.}
Now taking $q=\frac{d}{2}$ and $r=d$ in \eqref{tuexp}, then for any $p>1$ and $\frac d2\leq p\leq d$, we deduce from \eqref{tuexp}  and  Gronwall's inequality that
\begin{align}\label{t1a}
	\|\u(t)\|_{L^p(\Omega)}\leq ct^{-\frac{d}{2}(\frac 2d-\frac1p)} ( \|u_0\|_{L^{d/2}(\Omega)}+\|\nabla v_0\|_{L^d(\Omega)})\qquad\text{for} \;t\leq1
\end{align} where $c$ depends on $\Omega$. It follows that for $t\leq1,$
\begin{align}
	\|\nabla v(t)\|_{L^d(\Omega)}
	\leq&\|\nabla e^{t(\Delta-1)}v_0\|_{L^d(\Omega)}+\int_0^t\|\nabla e^{(t-s)(\Delta-1)}\u(s)\|_{L^d(\Omega)}ds\nonumber\\
	\leq& 2k_3\|\nabla v_0\|_{L^d(\Omega)}+k_2\int_0^te^{-(\lambda_1+1)(t-s)}(1+(t-s)^{-\frac12})\|\u(s)\|_{L^d(\Omega)}ds\nonumber\\
	\leq& 2k_3\|\nabla v_0\|_{L^d(\Omega)}\non\\
	&+ck_2 ( \|u_0\|_{L^{d/2}(\Omega)}+\|\nabla v_0\|_{L^d})\int_0^te^{-(\lambda_1+1)(t-s)}(1+(t-s)^{-\frac12})s^{-\frac{1}{2}}ds\nonumber\\
	\leq& c(\|u_0\|_{L^{d/2}(\Omega)}+\|\nabla v_0\|_{L^d(\Omega)}).\label{347}
\end{align}

When $d<p<\infty$, thanks to \eqref{estpl0}, \eqref{estpl}, \eqref{t1a} and \eqref{347}, we infer that
\begin{align}
	\|\u(t)\|_{L^p(\Omega)}\leq& c(t/2)^{-\frac d2(\frac1{d}-\frac1p)}(\|\u(\frac{t}{2}))\|_{L^{d}(\Omega)}+\|\nabla \v(\frac{t}{2}))\|_{L^{d}(\Omega)})\nonumber\\
	\leq&c(t/2)^{-\frac d2(\frac1{d}-\frac1p)}(t/2)^{-\frac d2(\frac2{d}-\frac1{d})}(\|u_0\|_{L^{d/2}(\Omega)}+\|\nabla v_0\|_{L^{d}(\Omega)})\nonumber\\
	\leq&ct^{-\frac{d}{2}(\frac 2d-\frac1p)} ( \|u_0\|_{L^{d/2}(\Omega)}+\|\nabla v_0\|_{L^d(\Omega)})\qquad\text{for} \;t\leq1
\end{align}with $c$ depends on $d$ and $\Omega$.
 Summing up, we  conclude that \eqref{t1a} holds for any $p>1$ satisfying $\frac{d}{2}\leq p< \infty$ with $c$ depends on $\Omega$ and $d$ at most.

As a consequence, for $t\leq1$ and any $d\leq p< \infty,$\begin{align}\label{t1b}
	\|\nabla \v(t)\|_{L^p(\Omega)}\leq&\|\nabla e^{t(\Delta-1)}v_0\|_{L^p(\Omega)}+\int_0^t\|\nabla e^{(t-s)(\Delta-1)}\u(s)\|_{L^p(\Omega)}ds\nonumber\\
	\leq& 2k_3t^{-\frac{d}{2}(\frac1d-\frac1p)}\|\nabla v_0\|_{L^d(\Omega)}+k_2\int_0^te^{-(\lambda_1+1)(t-s)}(1+(t-s)^{-\frac12})\|\u(s)\|_{L^p(\Omega)}ds\nonumber\\
	\leq& 2k_3t^{-\frac{d}{2}(\frac1d-\frac1p)}\|\nabla v_0\|_{L^d(\Omega)}\non\\
	&+ck_2 ( \|u_0\|_{L^{d/2}(\Omega)}+\|\nabla v_0\|_{L^d})\int_0^te^{-(\lambda_1+1)(t-s)}(1+(t-s)^{-\frac12})s^{-\frac{d}{2}(\frac2d-\frac1p)}ds\nonumber\\
	\leq& c(1+p)t^{-\frac{d}{2}(\frac1d-\frac1p)}(\|u_0\|_{L^{d/2}(\Omega)}+\|\nabla v_0\|_{L^d(\Omega)})
\end{align}with  $c$ depends on $\Omega$ and $d$, since by Lemma \ref{lmint}, when $t\leq1$
\begin{align}
	&\int_0^te^{-(\lambda_1+1)(t-s)}(1+(t-s)^{-\frac12})s^{-\frac{d}{2}(\frac2d-\frac1p)}ds\nonumber\\
	\leq& e^{\lambda_1}\int_0^te^{-(\lambda_1+1)(t-s)}(1+(t-s)^{-\frac12})s^{-\frac{d}{2}(\frac2d-\frac1p)}e^{-\lambda_1s}ds\nonumber\\
	\leq &\frac{cp}{d}\non\\
\leq&\frac{cp}{d}t^{-\frac d2(\frac1d-\frac1p)}.
\end{align}

%%%%%%%%%%%%%%%%%%%%%%%%%%%%%%%%%%%%%%%%%%%%%%%%%%%%%%%%%%%%%%
%%%%%%%%%%%%Step 6
%%%%%%%%%%%%%%s%%%%%%%%%%%%%%%%%%%%%%%%%%%%%%%%%%%%%%%%%%%%%%

\textbf{Step 6.} Now, for $t\geq1$, when $p<2$, using H\"older's inequality and Lemma \ref{lmexdecaypp}, we find
\begin{align}\label{t2a}
	\|\u(t)\|_{L^p(\Omega)}+\sqrt{\M}\|\nabla \v(t)\|_{L^p(\Omega)}\leq&c |\Omega|^{\frac1p-\frac12}\bigg(\|\u(t)\|_{L^2(\Omega)}+\sqrt{\M}\|\nabla \v(t)\|_{L^2(\Omega)}\bigg)\non\\
	\leq&c |\Omega|^{\frac1p-\frac12}e^{-\mu_1(t-\frac12)}\left(\|\u(\frac12)\|_{L^2(\Omega)}+\sqrt{\M}\|\nabla \v(\frac12)\|_{L^2(\Omega)}\right)\nonumber\\
	\leq&c |\Omega|^{\frac1p-\frac12}|\Omega|^{\frac12-\frac1d}e^{-\mu_1(t-\frac12)}\left(\|\u(\frac12)\|_{L^d(\Omega)}+\sqrt{\M}\|\nabla \v(\frac12)\|_{L^d(\Omega)}\right)\nonumber\\
	\leq& c(1+\sqrt{\M})e^{-\mu_1t}\bigg(\|u_0\|_{L^{d/2}(\Omega)}+\|\nabla v_0\|_{L^d(\Omega)}\bigg)
\end{align}
with $c$ depends on $d$ and $\Omega$ at most, and when $2\leq p<\infty$, we derive  from \eqref{estpl} and Lemma \ref{lmexdecaypp} that
\begin{align}\label{t2b}
		\|\u(t)\|_{L^p(\Omega)}+\|\nabla \v(t)\|_{L^p(\Omega)}\leq & c\left(\|\u(t-\frac12)\|_{L^2(\Omega)}+\|\nabla \v(t-\frac12)\|_{L^2(\Omega)}\right)\nonumber\\
		\leq & \frac{c}{\sqrt{\M}}\left(\|\u(t-\frac12)\|_{L^2(\Omega)}+\sqrt{\M}\|\nabla \v(t-\frac12)\|_{L^2(\Omega)}\right)\nonumber\\
		\leq& \frac{c}{\sqrt{\M}}e^{-\mu_1(t-1)}\left(\|\u(\frac12)\|_{L^2(\Omega)}+\sqrt{\M}\|\nabla \v(\frac12)\|_{L^2(\Omega)}\right)\nonumber\\
		\leq& \frac{c}{\sqrt{\M}}|\Omega|^{\frac12-\frac1d}e^{-\mu_1(t-1)}\left(\|\u(\frac12)\|_{L^d(\Omega)}+\sqrt{\M}\|\nabla \v(\frac12)\|_{L^d(\Omega)}\right)\nonumber\\
		\leq & \frac{c}{\sqrt{\M}}(1+\sqrt{\M})e^{-\mu_1t}\left(\|u_0\|_{L^{d/2}(\Omega)}+\|\nabla v_0\|_{L^d(\Omega)}\right)
\end{align}with $c$ depends on $d$ and $\Omega$ only. Finally, we may conclude the proof by combining \eqref{t1a}, \eqref{t1b}, \eqref{t2a} and \eqref{t2b}.
\end{proof}
\begin{remark}\label{remark}
Under the assumption of Lemma \ref{keylem1}, for any $2\leq l\leq p<\infty$,  there holds	\begin{equation}\label{lprem}
	\|\u(t)\|_{L^p(\Omega)}+\|\nabla\v(t)\|_{L^p(\Omega)}\leq ce^{-\mu_1t}(1+t^{-\frac d2(\frac1l-\frac1p)})(\|u_0\|_{L^{l}(\Omega)}+\|\nabla v_0\|_{L^l(\Omega)})
\end{equation}with $c$ depends on $d$, $\M$ and $\Omega$ only.
\end{remark}

For the special case $u_0=\nabla\cdot w$ and $v_0=0$, we have the following
\begin{lemma}
\label{Cor}Assume $d\geq2$ and $0<\M<1+\lambda_1.$ Suppose $u_0=\nabla\cdot w$ and $v_0=0$. Then there holds
	\begin{equation}
		\|\tilde{u}(t)\|_{L^p(\Omega)}\leq c_8e^{-\mu_1t}(1+t^{-\frac12-\frac{d}{2}(\frac1q-\frac1p)})\|w\|_{L^{q}(\Omega)}
	\end{equation} for any $\frac d2< q\leq p<\infty$, where $c_8>0$ depends on $d$, $\M$ and $\Omega$  if $d\geq3$ and also depends on $1/|q-1|$ if $d=2$. 	
\end{lemma}
\begin{proof}Under our assumption, for any $t\leq1$, in the same way as before,  we infer for any  $1< q\leq p<\infty$ that
\begin{align}\label{tuexpaaa}
	&\|\u(t)\|_{L^p(\Omega)}
\leq 2k_4 t^{-\frac12-\frac{d}{2}(\frac1q-\frac1p)}\|w\|_{L^{q}(\Omega)}\non\\
	&+k_2k_4\M\int_0^t\int_0^se^{-\lambda_1(t-s)}(1+(t-s)^{-\frac12})e^{-(\lambda_1+1)(s-\tau)}(1+(s-\tau)^{-\frac12})\|\u(\tau)\|_{L^p(\Omega)}d\tau ds.
\end{align} Letting $y(t)=t^{\frac12+\frac{d}{2}(\frac1q-\frac1p)}\|\u(t)\|_{L^q(\Omega)}$, we find that for $t\leq1$,
\begin{align}\label{corpq0}
	&y(t)\non\\
	\leq &c\|w\|_{L^q(\Omega)}\non\\
	+&c\M\int_0^t\int_0^se^{-\lambda_1(t-s)}e^{-(\lambda_1+1)(s-\tau)}(1+(t-s)^{-\frac12})(1+(s-\tau)^{-\frac12})\tau^{-\frac12-\frac{d}{2}(\frac1q-\frac1p)}y(\tau)d\tau ds\nonumber\\
	=& c \|w\|_{L^q(\Omega)}\non\\
	+&c\M\int_0^t\bigg(\int_\tau^te^{-\lambda_1(t-s)}e^{-(\lambda_1+1)(s-\tau)}(1+(t-s)^{-\frac12})(1+(s-\tau)^{-\frac12})ds\bigg)\tau^{-\frac12-\frac{d}{2}(\frac1q-\frac1p)}y(\tau)d\tau,
\end{align}where by Lemma \ref{lmint} and  direct calculations that for $\frac1q-\frac1p<\frac1{d}$ and for $t\leq1$,
\begin{align}
	&\int_0^t\bigg(\int_\tau^te^{-\lambda_1(t-s)}e^{-(\lambda_1+1)(s-\tau)}(1+(t-s)^{-\frac12})(1+(s-\tau)^{-\frac12})ds\bigg)\tau^{-\frac12-\frac{d}{2}(\frac1q-\frac1p)}d\tau\nonumber\\
	\leq &c\int_0^t\tau^{-\frac12-\frac{d}{2}(\frac1q-\frac1p)}e^{-\lambda_1(t-\tau)}d\tau\nonumber\\
	\leq&c\int_0^1\tau^{-\frac12-\frac{d}{2}(\frac1q-\frac1p)}d\tau\nonumber\\
	\leq& \frac{c}{\frac12-\frac d2(\frac1q-\frac1p)}
\end{align}with $c$ depends on $\Omega$ only. As a result, we deduce from Gronwall's inequality that for $t\leq1$ and  $\frac1q-\frac1p<\frac1{2d}$,
\begin{equation}\label{t1aa}
	\|\u(t)\|_{L^p(\Omega)}\leq ct^{-\frac12-\frac d2(\frac1q-\frac1p)}\|w\|_{L^q(\Omega)}
\end{equation}
with $c$ depends on  $\Omega$.  It follows that for $t\leq1$ and  $\frac1q-\frac1p<\frac1{2d}$
\begin{align}\label{t1bb}
	\|\nabla\v(t)\|_{L^p(\Omega)}\leq&\int_0^t\|\nabla e^{(t-s)(\Delta-1)}\u(s)\|_{L^p(\Omega)}ds\nonumber\\
	\leq& k_2\int_0^te^{-(\lambda_1+1)(t-s)}(1+(t-s)^{-\frac12})\|\u(s)\|_{L^p(\Omega)}ds\nonumber\\
	\leq& ck_2 \|w\|_{L^{q}(\Omega)}\int_0^te^{-(\lambda_1+1)(t-s)}(1+(t-s)^{-\frac12})s^{-\frac12-\frac{d}{2}(\frac1q-\frac1p)}ds\nonumber\\
	\leq& ct^{-\frac12-\frac{d}{2}(\frac1q-\frac1p)}\|w\|_{L^{q}(\Omega)},\end{align}
since for $t\leq1$, there holds\begin{align}
	&\int_0^te^{-(\lambda_1+1)(t-s)}(1+(t-s)^{-\frac12})s^{-\frac12-\frac{d}{2}(\frac1q-\frac1p)}ds\nonumber\\
	\leq&e^{\lambda_1}\int_0^te^{-(\lambda_1+1)(t-s)}(1+(t-s)^{-\frac12})s^{-\frac12-\frac{d}{2}(\frac1q-\frac1p)}e^{-\lambda_1s}ds\nonumber\\
	\leq&c(1+t^{-\frac{d}{2}(\frac1q-\frac1p)})\nonumber\\
	\leq&ct^{-\frac12-\frac{d}{2}(\frac1q-\frac1p)}.
\end{align}
On the other hand, note that our assumption ensures $0\leq\frac1q-\frac1p<\frac2d$. Therefore, for the case $\frac1q-\frac1p\geq\frac1{2d}$, we may split $(\frac1p,\frac1q)$ into four parts evenly with endpoints denoted by $\frac1p=\frac1{l_0}<\frac1{l_1}<...<\frac1{l_4}=\frac1q$ such that $0<\frac{1}{l_{k+1}}-\frac1{l_k}<\frac{1}{2d}$. 

Now, we divide our discussion into three cases regarding the dimensions. First, if $d\geq4$, there holds $2<q\leq p<\infty$. Then we may using \eqref{estpl} to derive that
\begin{align}
		\|\u(t)\|_{L^p(\Omega)}+\|\nabla \v(t)\|_{L^p(\Omega)}
		\leq&  c(t/2)^{-\frac d2(\frac1{l_3}-\frac1{p})}(\|\u(\frac {t}2)\|_{L^{l_3}(\Omega)}+\|\nabla \v(\frac {t}2)\|_{L^{l_3}(\Omega)})\nonumber\\
	\leq& c(t/2)^{-\frac d2(\frac1{l_3}-\frac1p)}(t/2)^{-\frac12-\frac d2(\frac1q-\frac1{l_3})}\|w\|_{L^q(\Omega)}\nonumber\\
	\leq & 	ct^{-\frac12-\frac d2(\frac1q-\frac1p)}\|w\|_{L^q(\Omega)}.
\end{align} Secondly, if $d=3$ and $\frac1q-\frac1p\geq\frac1{2d}=\frac16$, we must have $p>2$ since $q>\frac d2=\frac32$. If $2< q\leq p$, the proof is the same as above. If $1<q\leq2<p$, we infer  by the fact $\frac1q-\frac12<\frac23-\frac12=\frac16=\frac1{2d}$ that
\begin{align}
		\|\u(t)\|_{L^p(\Omega)}+\|\nabla \v(t)\|_{L^p(\Omega)}\leq&  c(t/2)^{-\frac d2(\frac1{2}-\frac1p)}(\|\u(\frac {t}2)\|_{L^{2}(\Omega)}+\|\nabla \v(\frac {t}2)\|_{L^{2}(\Omega)})\nonumber\\
	\leq& ct^{-\frac d2(\frac1{2}-\frac1p)} t^{-\frac12-\frac d2(\frac1q-\frac1{2})}\|w\|_{L^q(\Omega)}\nonumber\\
	\leq & 	ct^{-\frac12-\frac d2(\frac1q-\frac1p)}\|w\|_{L^q(\Omega)}.
\end{align}Lastly, if $d=2$, we have $1<q\leq p<\infty$. The cases $2<q\leq p$ and $\frac43<q\leq 2<p$ can be dealt with in the same way as above. It remains to consider the case $\frac12<\frac1p<\frac34<\frac1q<1$ and the case $\frac1p\leq\frac12<\frac34<\frac1q<1.$ In the former case, we  infer from \eqref{corpq0} and Gronwall's inequality that
\begin{align}
		\|\u(t)\|_{L^p(\Omega)}\leq& 	ct^{-\frac12-\frac d2(\frac1q-\frac1p)}\exp\{\frac{c\M}{\frac12-\frac1q+\frac1p}\}\|w\|_{L^q(\Omega)}\nonumber\\
		\leq& ct^{-\frac12-\frac d2(\frac1q-\frac1p)}\exp\{\frac{c\M}{1-\frac1q}\}\|w\|_{L^q(\Omega)}\nonumber\\
		\leq& c_9t^{-\frac12-\frac d2(\frac1q-\frac1p)}\|w\|_{L^q(\Omega)}
\end{align}where $c_9>0$ depends on $d$, $\Omega$ and $1/|q-1|$.
For the latter case, we infer from above, \eqref{estpl} \eqref{corpq0} and \eqref{t1bb} that
\begin{align}
		\|\u(t)\|_{L^p(\Omega)}\leq& c(t/2)^{-\frac d2(\frac1{2}-\frac1p)}(\|\u(\frac {t}2)\|_{L^{2}(\Omega)}+\|\nabla \v(\frac {t}2)\|_{L^{2}(\Omega)})\nonumber\\
				\leq&c(t/2)^{-\frac d2(\frac1{2}-\frac1p)}(t/2)^{-\frac12-\frac d2(\frac1q-\frac12)}(1+\frac{c\M}{\frac12-\frac1q+\frac12})\exp\{\frac{c\M}{\frac12-\frac1q+\frac12}\}\|w\|_{L^q(\Omega)}\nonumber\\
		\leq& ct^{-\frac12-\frac d2(\frac1q-\frac1p)}(1+\frac{c\M}{1-1/q})\exp\{\frac{c\M}{1-1/q}\}\|w\|_{L^q(\Omega)}\nonumber\\
		\leq& c_9t^{-\frac12-\frac d2(\frac1q-\frac1p)}\|w\|_{L^q(\Omega)}
\end{align}where $c_9>0$ depends on $d$, $\Omega$ and $1/|q-1|$.

Summing up, we obtain that for $t\leq 1$ and for any $\max\{1,\frac d2\}< q\leq p<\infty$, there holds 
\begin{equation}\label{t1ccc}
		\|\u(t)\|_{L^p(\Omega)}+\|\nabla \v(t)\|_{L^p(\Omega)}\leq ct^{-\frac12-\frac d2(\frac1q-\frac1p)}\|w\|_{L^q(\Omega)}\end{equation}with  $c$ depends on $d$ and  $\Omega$ if $d\geq3$ and also depends on $1/|q-1|$ if $d=2.$

Now for $t\geq1$, using H\"older's inequality and Lemma \ref{lmexdecaypp}, we can prove in the same way as before that
\begin{align}\label{t2aa}
	\|\u(t)\|_{L^p(\Omega)}+\sqrt{\M}\|\nabla \v(t)\|_{L^p(\Omega)}	\leq ce^{-\mu_1t}\|w\|_{L^{q}(\Omega)}
\end{align}
with $c$ depends on $d$, $\M$ and  $\Omega$ if $d\geq3$ and also depends on $1/|q-1|$ if $d=2.$ Now, our assertion follows from  \eqref{t1ccc} and \eqref{t2aa}.

\end{proof}
%%%%%%%%%%%%%%%%%%%%%%%%%%%%%%%%%%%%%%%%%%%%%%%%%%%%%%%%%%%%%%%%

\section{Proof of Theorem \ref{TH0} and Theorem \ref{TH1}}
With the key  $L^p-L^q$ decay estimates established in Lemmas \ref{lmdecay}--\ref{cor2} and Lemmas \ref{keylem1}--\ref{Cor}, it remains to complete our proof by an adaptation of the one-step contradiction argument from \cite{Win10,Cao} or using the implicit function theory as done in \cite{KMS16}. In this paper, we choose the former way for convenience since we already have Lemma \ref{loex1} at hand. The main idea is to compute the difference  between the solution  of the nonlinear problem and the one to the corresponding linearized problem. Since the nonlinear problem can be regarded as its linearized one with a quadratic perturbation, with small initial data, the difference between their solutions should also be of a quadratic order.
The major difference is now $\M>0$, we have to compare the associated nonlinear semigroup with the linearized ones $e^{t\mathcal{L}}$ or $e^{t\mathcal{A}}$ while when $\M=0$, we only have to compute its difference with $e^{t\Delta}$ whose decay behavior is well-known as shown in Lemma \ref{lmpq}.

Now, we report the proof in detail as follows.
\begin{proposition}\label{prop0}
	Suppose $d\geq2$ and $0<\M<1+\lambda_1$. For any fixed $q_0\in(\frac{d}{2},d)$ and $\lambda'< \mu_0$, there exists $\varepsilon_0>0$ depending on $d$, $\Omega$, $q_0$ and $\lambda'$ such that for any initial datum $u_0\in C(\overline{\Omega})\cap L^1_0(\Omega)$ satisfying $u_0+\M\geq0$ and $\|u_0\|_{L^{d/2}(\Omega)}\leq \varepsilon$ for some $\varepsilon<\varepsilon_0$, problem  \eqref{chemo2} with $\gamma=0$ has global classical solution which is globally bounded and satisfies
	\begin{equation}
		\|u(t)-e^{t\L}u_0\|_{L^\theta(\Omega)}\leq \varepsilon e^{-\lambda't}(1+t^{-1+\frac{d}{2\theta}}),\qquad \forall \;t>0.
	\end{equation}for all $\theta\in[q_0,\infty]$.
\end{proposition}
\begin{proof}According to Lemma \ref{loex1}, problem \eqref{chemo1} with $\gamma=0$ and nonnegative initial data $\rho_0=u_0+\M$  has a unique classical solution on $[0,T_{\mathrm{max}})$ and if $T_{\mathrm{max}}<\infty,$ we have $\limsup\limits_{t\nearrow T_{\mathrm{max}}}\|\rho(\cdot,t)\|_{L^\infty}=\infty.$ Therefore, for problem \eqref{chemo2} with $\gamma=0$, we obtain a classical solution $(u,v)=(\rho-\M,c-\M)$ on $[0,T_{\mathrm{max}})$ and if $T_{\mathrm{max}}<\infty,$ we have $\limsup\limits_{t\nearrow T_{\mathrm{max}}}\|u(\cdot,t)\|_{L^\infty(\Omega)}=\infty.$

For fixed $\frac d2<q_0<d$ and $d<\theta_0<\frac{dq_0}{d-q_0}$, we set 
\begin{equation}\nonumber
	T_0\triangleq\sup\bigg{\{}T>0:\|u(t)-e^{t\mathcal{L}}u_0\|_{L^\theta}\leq\varepsilon e^{-\lambda't}(1+t^{-1+\frac{d}{2\theta}}),\;\;\text{for all}\;t\in[0,T)\;\text{and all} \;\theta\in[q_0,\infty].\bigg{\}}
\end{equation}
Then $T_0$ is well-defined  and positive with $T_0\leq T_{\mathrm{max}}$, because both $u(t)$ and $e^{t\L}u_0$ are bounded near $t=0$ due to Lemma \ref{loex1} and Lemma \ref{lmdecay}, while on the other hand as $t\rightarrow0^+,$ $t^{-1+\frac{d}{2\theta}}\geq t^{-1+\frac{d}{2q_0}}\rightarrow+\infty$ uniformly with respect to $\theta\in[q_0,\infty]$. Now we claim that when $\varepsilon_0$ is sufficiently small, we have $T_0=\infty.$
First, we observe that by Lemma \ref{lmdecay},
\begin{align}
	\|u(t)\|_{L^\theta(\Omega)}\leq& \|u(t)-e^{t\mathcal{L}}u_0\|_{L^\theta(\Omega)}+\|e^{t\mathcal{L}}u_0\|_{L^\theta(\Omega)}\nonumber\\
	\leq& \varepsilon (1+t^{-1+\frac{d}{2\theta}})e^{-\lambda't}+c_1(1+t^{-1+\frac{d}{2\theta}})e^{-\mu_0t}\|u_0\|_{L^{d/2}(\Omega)}\nonumber\\
	\leq& C_1\varepsilon (1+t^{-1+\frac{d}{2\theta}})e^{-\lambda't}
\end{align} holds for all $0<t<T_0$ where  $C_1$ depends on $d$ and $\Omega$. On the other hand, by Sobolev embedding theorem and the classical theory for elliptic equations,  we have
\begin{align}\label{embed}
	\|\nabla v(t)\|_{L^{q_2}(\Omega)}\leq C_2\| v(t)\|_{W^{2,q_0}(\Omega)}\leq C_2\|u(t)\|_{L^{q_0}(\Omega)}
\end{align}with $\frac1{q_0}=\frac1d+\frac1{q_2}$  and $C_2$ depends on $d$, $\Omega$ and $q_0$.

 Due to the variation-of-constants formula, there holds
	\begin{align}\label{express1}
		u(t)-e^{t\L}u_0=-\int_0^t e^{(t-s)\L}\nabla\cdot(u(s)\nabla v(s))ds.
\end{align}
Now  we  estimate the term on right hand-side of \eqref{express1}. Invoking Lemma \ref{cor2}, for $\theta\in(\theta_0,\infty]$, there holds 
\begin{align}
	&\|u(t)-e^{t\L}u_0\|_{L^\theta(\Omega)}\non\\
	\leq&\bigg{\|}\int_0^t e^{(t-s)\L}\nabla\cdot(u\nabla v)(s)ds\bigg{\|}_{L^\theta(\Omega)}\nonumber\\
	\leq& c_5\int_0^te^{-\mu_0 (t-s)}(1+(t-s)^{-\frac12-\frac d2(\frac{1}{\theta_0}-\frac{1}{\theta})})\big{\|}u\nabla v\big{\|}_{L^{\theta_{0}}(\Omega)}ds\nonumber\\
	\leq&c_5\int_0^te^{-\mu_0 (t-s)}(1+(t-s)^{-\frac12-\frac d2(\frac{1}{\theta_0}-\frac{1}{\theta})})\|u(s)\|_{L^{q_1}(\Omega)}\|\nabla v(s)\|_{L^{q_2}(\Omega)}ds\nonumber\\
	\leq&c_5C_2\int_0^te^{-\mu_0 (t-s)}(1+(t-s)^{-\frac12-\frac d2(\frac{1}{\theta_0}-\frac{1}{\theta})})\|u(s)\|_{L^{q_1}(\Omega)}\|u(s)\|_{L^{q_0}(\Omega)}ds\nonumber\\
	\leq& c_5C_1C_2\int_0^te^{-\mu_0 (t-s)}(1+(t-s)^{-\frac12-\frac d2(\frac{1}{\theta_0}-\frac{1}{\theta})})\varepsilon(1+s^{-1+\frac{d}{2q_1}}) e^{-\lambda's}\varepsilon(1+s^{-1+\frac{d}{2q_0}}) e^{-\lambda's}ds\nonumber\\
	\leq &c_5C_1C_2\varepsilon^2\int_0^te^{-\mu_0(t-s)}e^{-\lambda's}(1+(t-s)^{-\frac12-\frac d2(\frac{1}{\theta_0}-\frac{1}{\theta})})(1+s^{-\frac32+\frac{d}{2\theta_0}})ds\nonumber\\
	\leq& C_3\varepsilon^2 e^{-\lambda't}(1+t^{-1+\frac{d}{2\theta}})
\end{align}for all $t\in[0,T_0)$  with $C_3>0$ depends on $d$, $\Omega$, $|\mu_0-\lambda'|$, $q_0$ and $\theta_0$, since under our assumption we can find $q_1,q_2 \geq q_0$ such that $\frac{1}{\theta_0}=\frac{1}{q_1}+\frac{1}{q_2}$ and $\frac{1}{q_0}=\frac{1}{q_2}+\frac{1}{d}$.

On the other hand, for $\theta\in[q_0,\theta_0]$, we infer that
\begin{align}
	&\|u(t)-e^{t\L}u_0\|_{L^\theta(\Omega)}\non\\
	\leq&\bigg{\|}\int_0^t e^{(t-s)\L}\nabla\cdot(u\nabla v)(s)ds\bigg{\|}_{L^\theta(\Omega)}\nonumber\\
	\leq& c_5\int_0^te^{-\mu_0 (t-s)}(1+(t-s)^{-\frac12-\frac d2(\frac{1}{q_0}-\frac{1}{\theta})})\big{\|}u\nabla v\big{\|}_{L^{q_{0}}(\Omega)}ds\nonumber\\
	\leq&c_5\int_0^te^{-\mu_0 (t-s)}(1+(t-s)^{-\frac12-\frac d2(\frac{1}{q_0}-\frac{1}{\theta})})\|u(s)\|_{L^{d}(\Omega)}\|\nabla v(s)\|_{L^{q_2}(\Omega)}ds\nonumber\\
	\leq&c_5C_2\int_0^te^{-\mu_0 (t-s)}(1+(t-s)^{-\frac12-\frac d2(\frac{1}{q_0}-\frac{1}{\theta})})\|u(s)\|_{L^{d}(\Omega)}\|u(s)\|_{L^{q_0}(\Omega)}ds\nonumber\\
	\leq& c_5C_1C_2\int_0^te^{-\mu_0 (t-s)}(1+(t-s)^{-\frac12-\frac d2(\frac{1}{q_0}-\frac{1}{\theta})})\varepsilon(1+s^{-\frac12}) e^{-\lambda's}\varepsilon(1+s^{-1+\frac{d}{2q_0}}) e^{-\lambda's}ds\nonumber\\
	\leq &c_5C_1C_2\varepsilon^2\int_0^te^{-\mu_0(t-s)}e^{-\lambda's}(1+(t-s)^{-\frac12-\frac d2(\frac{1}{q_0}-\frac{1}{\theta})})(1+s^{-\frac32+\frac{d}{2q_0}})ds\nonumber\\
	\leq& C_3\varepsilon^2 e^{-\lambda't}(1+t^{-1+\frac{d}{2\theta}})
\end{align}for all $t\in[0,T_0)$.  

As a result, we conclude that for all $\theta\in[q_0,\infty]$, there holds
\begin{equation}
	\|u(t)-e^{t\L}u_0\|_{L^\theta(\Omega)}\leq C_3\varepsilon^2 e^{-\lambda't}(1+t^{-1+\frac{d}{2\theta}})
\end{equation}for all $t\in[0,T_0)$ with $C_3>0$ depends on $d$, $\Omega$, $|\mu_0-\lambda'|$, $q_0$ and $\theta_0$, but is independent of $T_0$ or $t$.
Choosing $\varepsilon_0<\frac{1}{2C_3}$, we conclude that $T_0=\infty$ and hence $T_{\mathrm{max}}=\infty$ as well which completes the proof.
\end{proof}
Now we complete the proof of Theorem \ref{TH0}. We may fix $q_0\in(\frac d2,d)$, $\theta_0\in(d,\frac{dq_0}{d-q_0})$ and $\lambda'<\mu_0$, then by Proposition \ref{prop0}, we get $\varepsilon_0>0$ such that $(u,v)$ exists globally under the smallness assumptions. Recalling that $(u,v)$ is just a reduction of $\M$ from $(\rho,c)$ we conclude that under the assumption of Theorem \ref{TH0},  problem \eqref{chemo1} has a unique classical solution  $(\rho,c)$ that is globally bounded. Moreover, since $\|e^{t\L}u_0\|_{L^\infty}\leq C e^{-\mu_0 t}\|u_0\|_{L^\infty}$ for $t\geq1$ due to Lemma \ref{lmdecay}, we infer that for $t\geq1,$
\begin{align}
	\|u(t)\|_{L^\infty(\Omega)}\leq \|u(t)-e^{t\L}u_0\|_{L^\infty(\Omega)}+\|e^{t\L}u_0\|_{L^\infty(\Omega)}\leq Ce^{-\lambda't}
\end{align}which indicates that
\begin{equation}
	\|\rho(t)-\M\|_{L^\infty(\Omega)}+\|\nabla c(t)\|_{L^\infty(\Omega)}\leq Ce^{-\lambda't}
\end{equation}for $t\geq1$ with some $C>0$. This completes the proof of Theorem \ref{TH0}.

%%%%%%%%%%%%%%%%%%%%%%%%%%%%%%%%%%%%%%%%%%%%%%%%%%%%

Now, we prove Theorem \ref{TH1}. To this aim, we prove the following result for the reduced cell density and chemical concentration $(u,v)$ of \eqref{chemo2}.
\begin{proposition}
			Let $d\geq2$ and $0<\M<1+\lambda_1$. For any fixed $q_0,\theta_0>0$ such that $\frac d2<q_0<d$ and $d<\theta_0<\frac{dq_0}{d-q_0}$, there exists $\varepsilon_0>0$ depending on $d$, $q_0$, $\M$ and $\Omega$  such that for any initial data $(u_0,v_0)\in C(\overline{\Omega})\cap L^1_0(\Omega)\times C^1(\overline{\Omega})\cap L^1_0(\Omega)$ satisfying  $\partial_\nu v_0=0$ on $\partial\Omega$,  $u_0\geq-\M$, $v_0\geq-\M$ and $\|u_0\|_{L^{d/2}(\Omega)}+\|\nabla v_0\|_{L^d(\Omega)}\leq\varepsilon$ for some $\varepsilon<\varepsilon_0$, system \eqref{chemo2} with $\gamma=1$ has  global classical solution $(u,v)$ which is bounded such that
	\begin{equation}
		\|u(t)-\u(t)\|_{L^\theta(\Omega)}\leq\varepsilon e^{-\mu't }(1+t^{-1+\frac{d}{2\theta}})\qquad\text{for all}\;\;t>0
	\end{equation} and for all $\theta\in[q_0,\theta_0]$ with $\mu'<\mu_1=\lambda_1-\frac12\bigg(\sqrt{4\lambda_1\M+1}-1\bigg)>0.$

\end{proposition}
\begin{proof}According to Lemma \ref{loex1}, problem \eqref{chemo1} with nonnegative initial data $\rho_0=u_0+\M$ and $c_0=v_0+\M$ has a classical solution on $[0,T_{\mathrm{max}})$ and if $T_{\mathrm{max}}<\infty,$ we have $\limsup\limits_{t\nearrow T_{\mathrm{max}}}\|\rho(\cdot,t)\|_{L^\infty}=\infty.$ Therefore, for problem \eqref{chemo2}, we obtain a classical solution $(u,v)=(\rho-\M,c-\M)$ on $[0,T_{\mathrm{max}})$ and if $T_{\mathrm{max}}<\infty,$ we have $\limsup\limits_{t\nearrow T_{\mathrm{max}}}\|u(\cdot,t)\|_{L^\infty}=\infty.$
	Denoting $(\tilde{u},\tilde{v})$ the solution of linearized system \eqref{chemo3b}, we derive by variation-of-constants formula that for $0<t<T_{\mathrm{max}}$,
\begin{align}\label{vform}
	v(t)=&e^{t(\Delta-1)}v_0+\int_0^t e^{(t-s)(\Delta-1)}u(s)ds\nonumber\\
	= &e^{t(\Delta-1)}v_0+\int_0^t e^{(t-s)(\Delta-1)}\tilde{u}(s)ds+\int_0^t e^{(t-s)(\Delta-1)}\big(u(s)-\tilde{u}(s)\big)ds\nonumber\\
	=&\tilde{v}(t)+\int_0^t e^{(t-s)(\Delta-1)}\big(u(s)-\tilde{u}(s)\big)ds.
\end{align}%%%%%%%%%%%%%%%%%%%%%%%%%%%%%%%%%%%%

On the other hand, exploiting the semigroup $e^{t\mathcal{A}}$ and variation-of-constants formula again, we infer that
\begin{align}\label{uvexp}
	\left(\begin{matrix}u(t)\\
v(t)\end{matrix}\right)&=e^{t\mathcal{A}}	\left(\begin{matrix}u_0\\
v_0\end{matrix}\right)+\int_0^te^{(t-s)\mathcal{A}}	\left(\begin{matrix}-\nabla\cdot(u(s)\nabla v(s))\\
0\end{matrix}\right)ds
\end{align}
from which we represent $u$ according to
\begin{align}\label{uform}
	u(t)=\u(t)-\int_0^t\Phi_1^{t-s}\left(\nabla\cdot(u(s)\nabla v(s)),0\right)ds.
\end{align}

Now, like in \cite{Cao} (see also \cite{Win10}), fix some $q_0\in(\frac d2,d)$ and $\theta_0\in(d,\frac{dq_0}{d-q_0})$, one can choose $q_1,q_2>0$ such that $q_1\in(q_0,\theta_0)$, $q_2\in(d,\frac{dq_0}{d-q_0})$ and $\frac1{q_0}=\frac{1}{q_1}+\frac{1}{q_2}.$
Define
\begin{equation}\nonumber
	T_1\triangleq\sup\bigg{\{}T>0:\|u(t)-\tilde{u}(t)\|_{L^\theta}\leq\varepsilon e^{-\mu' t}(1+t^{-1+\frac{d}{2\theta}}),\;\;\text{for all}\;t\in[0,T)\;\text{and all} \;\theta\in[q_0,\theta_0].\bigg{\}}
\end{equation}
Then $T_1$ is well-defined and positive with $T_1\leq T_{\mathrm{max}}$, since on the one hand, $\|u(t)\|_{L^\infty}$ and $\|\u(t)\|_{L^{\theta_0}}$  are both  bounded near $t=0$ due to Lemma \ref{loex1} and Remark \ref{remark}. Hence $\|\u(t)\|_{L^\theta}\leq |\Omega|^{1/\theta-1/{\theta_0}}\|\u(t)\|_{L^{\theta_0}}\leq \max\{|\Omega|^{1/{q_0}},1\}\|\u(t)\|_{L^{\theta_0}}$ are uniformly bounded with respect to $\theta\in[q_0,\theta_0]$. On the other hand as $t\rightarrow0^+,$ $t^{-1+\frac{d}{2\theta}}\geq t^{-1+\frac{d}{2q_0}}\rightarrow+\infty$ uniformly with respect to $\theta\in[q_0,\theta_0]$. Now we claim that when $\varepsilon_0$ is sufficiently small, we have $T_1=\infty.$

First, we apply $\nabla$ to both sides of \eqref{vform} to deduce that
\begin{align}
	\|\nabla v(t)-\nabla \tilde{v}(t)\|_{L^p(\Omega)}\leq& \int_0^t\|\nabla e^{(t-s)(\Delta-1)}(u(s)-\tilde{u}(s))\|_{L^p(\Omega)}ds\nonumber\\
	\leq&\int_0^t k_2(1+(t-s)^{-\frac12-\frac{d}{2}(\frac1{q_0}-\frac1p)})e^{-(\lambda_1+1)(t-s)}\|u(s)-\tilde{u}(s)\|_{L^{q_0}(\Omega)}ds\nonumber\\
	\leq&\int_0^t k_2(1+(t-s)^{-\frac12-\frac{d}{2}(\frac1{q_0}-\frac1p)})e^{-(\lambda_1+1)(t-s)}\varepsilon e^{-\mu't}(1+s^{-1+\frac{d}{2q_0}})ds\nonumber\\
	\leq& Ck_2\varepsilon(1+t^{-\frac{1}{2}+\frac{d}{2p}})e^{-\mu't}\nonumber\\
	\leq& C_4\varepsilon(1+t^{-\frac{1}{2}+\frac{d}{2p}})e^{-\mu't}
\end{align}for all $t\in(0,T_1)$ and all $p\in[q_0,\frac{dq_0}{d-q_0})$ where $C_4>0$ depends on $q_0$ and $\Omega$ only.
As a result, we infer by Lemma \ref{keylem1} that for $p\in[d,\frac{dq_0}{d-q_0})$
\begin{align}
	\|\nabla v(t)\|_{L^p(\Omega)}\leq&\|\nabla \v(t)\|_{L^p(\Omega)}+C_4\varepsilon(1+t^{-\frac{1}{2}+\frac{d}{2p}})e^{-\mu't}\nonumber\\
	\leq& c_7p\varepsilon e^{-\mu_1t}(1+t^{-\frac12+\frac{d}{2p}})+C_4\varepsilon(1+t^{-\frac{1}{2}+\frac{d}{2p}})e^{-\mu't}\nonumber\\
	\leq &C\varepsilon e^{-\mu't}(1+t^{-\frac12+\frac{d}{2p}})
\end{align} where $C$ depends on $d$, $\M$, $q_0$ and $\Omega$. For $p\in[q_0,d)$, applying H\"older's inequality and Lemma \ref{keylem1}, there holds
\begin{align}
	\|\nabla v(t)\|_{L^p(\Omega)}\leq& |\Omega|^{\frac{1}{p}-\frac1d}\|\nabla \v(t)\|_{L^d(\Omega)}+C_4\varepsilon(1+t^{-\frac{1}{2}+\frac{d}{2p}})e^{-\mu't}\nonumber\\
	\leq&c_7d|\Omega|^{\frac1p-\frac1d}e^{-\mu_1t}\varepsilon+C_4\varepsilon(1+t^{-\frac{1}{2}+\frac{d}{2p}})e^{-\mu't}\nonumber\\
	\leq&C\varepsilon(1+t^{-\frac{1}{2}+\frac{d}{2p}})e^{-\mu't}
\end{align}with $C$ depends on $d$, $q_0$, $\M$ and $\Omega$ only. Thus, we conclude from above that for all $p\in[q_0,\frac{dq_0}{d-q_0})$
\begin{equation}\label{vp}
		\|\nabla v(t)\|_{L^p(\Omega)}\leq C_5\varepsilon(1+t^{-\frac{1}{2}+\frac{d}{2p}})e^{-\mu't}
\end{equation}with $C_5$ depends on $d$, $q_0$, $\M$ and $\Omega.$

On the other hand, in views of the definition of $T_1$ and Lemma \ref{keylem1}, we also have
\begin{align}\label{up}
	\|u(t)\|_{L^p(\Omega)}\leq& \|\u(t)\|_{L^p(\Omega)}+\varepsilon e^{-\mu't}(1+t^{-1+\frac{d}{2p}})\nonumber\\
	\leq& c_6e^{-\mu_1t}(1+t^{-1+\frac{d}{2p}})\varepsilon+\varepsilon e^{-\mu't}(1+t^{-1+\frac{d}{2p}})\nonumber\\
	\leq & C_6e^{-\mu't}(1+t^{-1+\frac{d}{2p}})\varepsilon
\end{align} holds for all $p\in[q_0,\theta_0]$ and $0<t<T_1$, where $C_6$ depends on $d$ and $\Omega$ only.

Now due to \eqref{uform} and Lemma \ref{Cor}, we can finally estimate 
\begin{align}
	&\|u(t)-\tilde{u}(t)\|_{L^\theta(\Omega)}\non\\
	\leq& \int_0^t\|\Phi_1^{t-s}\left(\nabla\cdot(u(s)\nabla v(s)),0\right)\|_{L^\theta(\Omega)}ds\nonumber\\
	\leq&c_8\int_0^te^{-\mu_1(t-s)}(1+(t-s)^{-\frac{1}{2}-\frac d2(\frac1{q_0}-\frac1\theta)})\|u(s)\nabla v(s)\|_{L^{q_0}(\Omega)}ds\nonumber\\
	\leq &c_8\int_0^te^{-\mu_1(t-s)}(1+(t-s)^{-\frac{1}{2}-\frac d2(\frac1{q_0}-\frac1\theta)})\|u(s)\|_{L^{q_1}(\Omega)}\|\nabla v(s)\|_{L^{q_2}(\Omega)}ds\nonumber\\
	\leq &c_8C_5C_6\varepsilon^2\int_0^te^{-\mu_1(t-s)}(1+(t-s)^{-\frac{1}{2}-\frac d2(\frac1{q_0}-\frac1\theta)})e^{-2\mu's}(1+s^{-1+\frac d{2q_1}})(1+s^{-\frac12+\frac d{2q_2}})ds\nonumber\\
	\leq& c_8C_5C_6\varepsilon^2\int_0^te^{-\mu_1(t-s)}(1+(t-s)^{-\frac{1}{2}-\frac d2(\frac1{q_0}-\frac1\theta)})e^{-\mu's}(1+s^{-\frac32+\frac d{2q_0}})ds\nonumber\\
	\leq& C_7\varepsilon^2(1+t^{-1+\frac d{2\theta}})e^{-\mu't}
\end{align}for all $0<t<T_1$,
with $C_7$ depends on $d$, $\M$, $q_0$ and $\Omega$, but is independent of $T_1$ or $t$. Now, choosing $\varepsilon_0<\frac{1}{2C_7}$, we conclude that $T_1=\infty$ which implies that $T_{\mathrm{max}}=\infty$ as well, i.e.,  $(u,v)$ is global and bounded.\end{proof}

By the same argument as before, keeping in mind that $(u,v)$ is just a reduction of $\M$ from $(\rho,c),$ we obtain the existence of a unique global solution to problem \eqref{chemo1} with $\gamma=1$ under the assumptions of Theorem \ref{TH1}. It remains to show the exponentially decay in $L^\infty-$norm. 
 Invoking \eqref{vp}, the fact $\|u(t)\|_{L^\infty}\leq C$ and Lemma \ref{lmint}, we deduce that
\begin{align}
	&\|u(t)\|_{L^\infty(\Omega)}\non\\
	\leq& \|e^{t\Delta}u_0\|_{L^\infty(\Omega)}+\M\int_0^t\|e^{\Delta(t-s)}\Delta v(s)\|_{L^\infty(\Omega)}ds+\int_0^t\|e^{\Delta(t-s)}\nabla\cdot(u(s)\nabla v(s))\|_{L^\infty(\Omega)}ds\nonumber\\
	\leq& k_1e^{-\lambda_1t}\|u_0\|_{L^\infty(\Omega)}+C\int_0^t(1+(t-s)^{-\frac12-\frac{d}{2p}})e^{-\lambda_1(t-s)}\|\nabla v(s)\|_{L^p(\Omega)}ds\nonumber\\
	&\quad+C\int_0^t(1+(t-s)^{-\frac12-\frac{d}{2p}})e^{-\lambda_1(t-s)}\|u(s)\nabla v(s)\|_{L^p(\Omega)}ds\nonumber\\
	\leq&Ce^{-\lambda_1t}+C\int_0^t(1+(t-s)^{-\frac12-\frac{d}{2p}})e^{-\lambda_1(t-s)}\|\nabla v(s)\|_{L^p(\Omega)}ds\nonumber\\
	\leq&Ce^{-\lambda_1t}+C\int_0^t(1+(t-s)^{-\frac12-\frac{d}{2p}})e^{-\lambda_1(t-s)}(1+s^{-\frac12+\frac{d}{2p}})e^{-\mu's}ds\nonumber\\
	\leq& Ce^{-\lambda_1t}+Ce^{-\mu't}\nonumber\\
	\leq& Ce^{-\mu't}
\end{align}where $p\in(d,\frac{dq_0}{d-q_0})$ such that $\frac12+\frac{d}{2p}\in(0,1)$ and $\frac12-\frac{d}{2p}\in(0,1).$ As a consequence,
\begin{align}
	&\|\nabla v(t)\|_{L^\infty(\Omega)}\non\\
	\leq&\|\nabla e^{t(\Delta-1)}v_0\|_{L^\infty(\Omega)}+\int_0^t\|\nabla e^{(\Delta-1)(t-s)}u(s)\|_{L^\infty(\Omega)}ds\nonumber\\
	\leq & Ce^{-(\lambda_1+1)t}(1+t^{-\frac12})\|v_0\|_{L^\infty(\Omega)}+C\int_0^t(1+(t-s)^{-\frac12-\frac{d}{2q}})e^{-(\lambda_1+1)(t-s)}\|u(s)\|_{L^q(\Omega)}ds\nonumber\\
	\leq&Ce^{-(\lambda_1+1)t}(1+t^{-\frac12})+C\int_0^t(1+(t-s)^{-\frac12-\frac{d}{2q}})e^{-(\lambda_1+1)(t-s)}e^{-\mu's}(1+s^{-1+\frac{d}{2q}})ds\nonumber\\
	\leq& Ce^{-(\lambda_1+1)t}(1+t^{-\frac12})+Ce^{-\mu't}(1+t^{-\frac12})
\end{align}with some $q>d$, which indicates that when $t\geq1,$
\begin{equation}
		\|\nabla v(t)\|_{L^\infty(\Omega)}\leq Ce^{-\mu't}.
\end{equation}This completes the proof of Theorem \ref{TH1}.

\end{document}